\documentclass{amsart}
\usepackage{amsmath,amssymb,url,mathrsfs,amsthm}
\usepackage{enumitem} 
\usepackage{color}
\usepackage[all]{xypic}
\usepackage{tikz}
\usepackage[left=3.5cm,right=3.5cm,top=2.5cm,bottom=2.5cm]{geometry}
\usepackage{amsaddr}

\numberwithin{equation}{section}
\newtheorem{theorem}{Theorem}[section]
\newtheorem{lemma}[theorem]{Lemma}
\newtheorem{proposition}[theorem]{Proposition}
\newtheorem{corollary}[theorem]{Corollary}

\theoremstyle{definition}
\newtheorem{definition}[theorem]{Definition}

\newtheorem{example}[theorem]{Example}

\newcommand{\alg}[1]{{\textbf{\upshape #1}}}  %
\newcommand{\vv}[1]{\mathsf {#1}}
\newcommand{\cc}[1]{{\mathcal{#1}}}

\newcommand{\sse}{\subseteq}

\newcommand{\WH}{\vv W\vv H}
\newcommand{\MV}{\vv M\vv V}

\makeatletter
\def\square{\RIfM@\bgroup\else$\bgroup\aftergroup$\fi
  \vcenter{\hrule\hbox{\vrule\@height.6em\kern.6em\vrule}\hrule}\egroup}
\makeatother
\newcommand{\smlcirc}{\raise3pt\hbox{\textrm{\circle{3.3}}}}

\newcommand{\one}{\textsf{\bf 1}}

\newcommand{\myfrac}[2]{\dfrac{#1}{\lower.5ex\hbox{$#2$}}}
\mathchardef\hu="0362

\newcommand{\freeW}{\mathcal{F}_{\mathsf{WH}}}
\newcommand{\freeMV}{\mathcal{F}_{\mathsf{MV}}}
\newcommand{\free}{\mathcal{F}}
\newcommand{\term}{\mathcal{T}}
\newcommand{\uno}{ {\bf 1}}
\newcommand{\zz}{ \mathbb{Z}}
\newcommand{\nn}{ \mathbb{N}}
\newcommand{\W}{ \mathcal{W}}
\newcommand{\M}{ \mathcal{M}}
\newcommand{\x}{ {\bf x}}

\newcommand{\luk}{\L ukasiewicz}
\newcommand{\rest}{\upharpoonright}
\newcommand{\Z}{\mathbb{Z}}
\newcommand{\scP}{\mathscr{P}}
\newcommand{\scQ}{\mathscr{Q}}
\newcommand{\lu}{\textbf{\L}^+}
\newcommand{\Fil}{{\rm Fil}}
\begin{document}
\title{The polyhedral geometry of Wajsberg hoops}
\author{Sara Ugolini}
\address{Artificial Intelligence Research Institute (IIIA), CSIC, Barcelona, Spain}
\email{sara@iiia.csic.es; saraugolini.phd@gmail.com}
\date{}
\maketitle
\begin{abstract}
We show that the category of finitely presented Wajsberg hoops with homomorphisms is dually equivalent to a particular subcategory of rational polyhedra with $\mathbb{Z}$-maps. We use the duality to provide a geometrical characterization of finitely generated projective and exact Wajsberg hoops. As applications, we study logical properties of the  positive fragment of \luk\ logic. We show that, while deducibility in the fragment is equivalent to deducibility among positive formulas in \luk\ logic, the same is not true for admissibility of rules. Moreover, we show that the unification type of Wajsberg hoops is nullary, while the exact unification type is unitary, therefore showing decidability of admissible rules in the fragment.
\end{abstract}
\section{Introduction}
Residuated structures arise naturally in the study of various algebraic systems, such as ideals of rings \cite{WD39} or lattice-ordered groups \cite{BCGJT}, and have acquired a particular relevance in algebraic logic. Most notably, they encompass the equivalent algebraic semantics in the sense of Blok-Pigozzi \cite{BP} of so-called \emph{substructural logics} \cite{GJKO}. The latter are a large class of logical systems including classical logic and several of the most interesting non-classical logics, such as intuitionistic logic, fuzzy logics, and relevance logics to name a few. Algebraizability \emph{\`a la} Blok-Pigozzi entails that the deducibility relation of an algebraizable logic can be characterized by means of the algebraic equational consequence of its equivalent algebraic semantics. Therefore, interesting properties of algebraizable logics can be studied via algebraic means (for details on algebraizability and a collection of general results on substructural logics see \cite{GJKO}). 

Hoops are a particular variety of residuated monoids related to logic which were defined in an unpublished manuscript by B\"uchi and Owens, inspired by the work of Bosbach on partially ordered monoids \cite{Bo69}.
 Hoops indeed correspond to commutative monoids endowed with a partial order $\leq$ definable by means of the residuum $\to$ of the monoid operation as $a \leq b$ if and only if $a \to b = 1$, and such an order is \emph{natural}, in the sense that $a \leq b$ if and only if there exists $c$ such that $a = b \cdot c$. The first systematic study of hoops has been carried out by Ferreirim in her PhD thesis \cite{Fe92} and in her works with Blok \cite{BF93,BlokFerr2000}, while in her work with Aglian\`o and Montagna \cite{AFM} they investigate the connection between hoops and fuzzy logics. 

Wajsberg hoops play a special role in this framework. From the algebraic point of view, they can be used to describe subdirectly irreducible hoops, and the whole variety of hoops can be obtained as the join of iterated powers of the variety of Wajsberg hoops, in the sense defined in \cite{BlokFerr2000}.
Wajsberg hoops also have a peculiar connection with lattice-ordered abelian groups (abelian $\ell$-groups for short). In fact, the variety $\WH$ of Wajsberg hoops is generated by its totally ordered members, that are, in loose terms, either negative cones of abelian $\ell$-groups, or \emph{intervals} of abelian $\ell$-groups \cite{AglianoPanti1999}. 

 Besides their purely algebraic interest, Wajsberg hoops are relevant structures in the realm of algebraic logic. Indeed, they are term-equivalent to the $0$-free subreducts of MV-algebras, the equivalent algebraic semantics of the infinite-valued \luk\ logic, seen as residuated lattices. Therefore, $\WH$ is the equivalent algebraic semantics of the positive (i.e., $0$-free) fragment of the infinite-valued \luk\ logic \textbf{\L}, the latter seen as an extension of the Full Lambek Calculus with exchange \textbf{FL$_e$}, (thus in the language with constants $0, 1$ and binary connectives $\land, \lor, \cdot, \to$, see \cite{GJKO}).
 
  We could actually rephrase the representation of totally ordered Wajsberg hoops in terms possibly more familiar to the reader expert in many-valued logics. Namely, totally ordered Wajsberg hoops are either cancellative hoops, i.e. hoops where the monoidal operation satisfies the cancellativity law, or MV-algebras, whose representation in terms of intervals of abelian $\ell$-groups with strong unit actually lifts to a categorical equivalence via the well-known Mundici's functor \cite{M86}.

In the context of the algebraic semantics of many-valued logics, the relevance of Wajsberg hoops is also related to the study of the equivalent algebraic semantics of H\'ajek Basic Logic and its positive subreducts (BL-algebras and basic hoops). Given the well-known decomposition result in terms of Wajsberg hoops for totally ordered BL-algebras given by Aglian\`o and Montagna \cite{AglianoMontagna2003}, and the representation of free BL-algebras in \cite{AB10} building on this result, the understanding of Wajsberg hoops is  key to obtain interesting results in this framework.

The categorical approach in studying varieties connected to Wajsberg hoops and MV-algebras has shown to be fruitful in finding interesting connections with categories of $\ell$-groups \cite{ACV,DL94,GT05,M86}. Moreover, both abelian $\ell$-groups and abelian $\ell$-groups with strong order unit (equivalently, MV-algebras), enjoy a deep connection with geometrical objects. In particular, finitely presented structures are dually equivalent, respectively, to categories whose objects are closed polyhedral cones \cite{Be77} and rational polyhedra \cite{MarraSpada13}. In the case of MV-algebras, the geometrical approach has allowed the investigation of logical properties of \luk\ logic (see for instance \cite{Mundici11,MarraSpada13}).

In this contribution we will show that finitely presented Wajsberg hoops have an interesting geometrical dual as well, in particular, with what we will call {\em pointed} rational polyhedra. More precisely, we show how finitely presented Wajsberg hoops are dually equivalent to a (non-full) subcategory of rational polyhedra with $\mathbb{Z}$-maps, given by rational polyhedra in unit cubes $[0,1]^{n}$ that contain the lattice point ${\bf 1}= (1, \ldots, 1)$, and {\em pointed} $\mathbb{Z}$-maps, that are $\mathbb{Z}$-maps that respect the lattice point ${\bf 1}$. In particular, we obtain this result by first demonstrating that Wajsberg hoops are equivalent to a subcategory of finitely presented MV-algebras, that we shall call in what follows \emph{positive}, and then suitably restricting the Marra-Spada duality \cite{MarraSpada13}.

The first geometrical approach to the study of Wajsberg hoops is in \cite{AglianoPanti1999}, which contains a key step to understand the geometrical duality, that is, the description of the free $n$-generated Wajsberg hoop as a particular subreduct of the free MV-algebra over the same number of generators. We show how this result can be seen as a consequence of both general properties of varieties of positive subreducts and specific  characteristics of MV-algebras.

The connection with the MV-algebraic duality with rational polyhedra will allow us to use the deep results obtained by Cabrer and Mundici about finitely generated projective MV-algebras \cite{CM09,CM12,Ca15} to describe finitely generated projective Wajsberg hoops. In particular, we show that no MV-algebra (or more precisely, its $0$-free reduct) is projective in the variety of Wajsberg hoops, and actually that finitely generated (nontrivial) projective Wajsberg hoops are necessarily unbounded. Interestingly enough, this implies that, in particular, the ($0$-free reduct of the) two-element Boolean algebra $\alg 2$ is not projective in the variety of residuated lattices, while $\alg 2$ is projective in every variety of bounded commutative integral residuated lattices, and in the variety of all bounded commutative integral residuated lattices it is the only finite projective algebra \cite{AU21}. 

The fact that Wajsberg hoops are the equivalent algebraic semantics of the positive fragment of \luk\ logic will allow us to use our algebraic and geometric investigation to derive some analogies and differences between \luk\ logic and its positive fragment, that we shall refer to as $\lu$. In particular, we show that while deducibility in $\lu$ coincides with deducibility of positive terms in \luk\ logic, the same is not true for admissibility of rules. 

Moreover, via the algebraic approach to unification problems developed by Ghilardi \cite{Gh97}, we will show that the unification type of the variety of Wajsberg hoops, and thus of the positive fragment of \L ukasiewicz logic, is nullary. This is in close analogy with the case of MV-algebras, and indeed our proof adapts to pointed rational polyhedra the pathological example given in \cite{MarraSpada13} for the case of \L ukasiewicz logic. 
Moreover, via the algebraic approach to admissibility developed in \cite{CMe15}, we show that while the exact unification type of \luk\ logic is finitary, the one of its positive fragment is unitary. This in particular implies decidability of the admissibility of rules in Wajsberg hoops and the positive fragment of \luk\ logic.

\section{Preliminaries}\label{sec:prel}
We start by recalling the needed facts about hoops, for more details and further references the reader can consult \cite{Fe92,BlokFerr2000,AFM}.

A \emph{hoop} is an algebra $\alg A = (A, \cdot, \to, 1)$ such that $(A, \cdot, 1)$ is a commutative monoid and:
\begin{enumerate}
	\item[(H1)] $x \to x = 1$,
	\item[(H2)] $x \cdot (x \to y) = y \cdot (y \to x)$,
	\item[(H3)] $(x \cdot y) \to z = x \to (y \to z)$.
\end{enumerate}
A \emph{Wajsberg hoop} is a hoop satisfying the following identity:
$$(x \to y) \to y = (y \to x) \to x \quad (T)$$
often called Tanaka's equation thanks to Tanaka's work in this framework \cite{Ta75}, although this particular identity seems to appear later in \cite{Yu77}.
Every hoop carries a natural order structure definable as:
$$x \leq y \mbox{ iff } x \to y = 1$$
that is in general a meet-semilattice order, where $\land$ is definable as:
$$x \land y = x \cdot (x \to y).$$
In the case of Wajsberg hoops this order is actually a lattice order, where the join operation is definable as:
$$x \lor y = (x \to y) \to y.$$
By Tanaka's identity, the defined join operation is commutative, and in fact in the literature on BCK-algebras (T) is also called the \emph{commutativity law}.

Moreover, $(\cdot, \to)$ form a \emph{residuated pair} with respect to this lattice order, that is to say:
$$x \cdot y \leq z \mbox{ iff } y \leq x \to z.$$
In fact, Wajsberg hoops can be seen as particular commutative residuated lattices. A \emph{commutative residuated lattice} is an algebra $\alg A = (A, \cdot, \to, \land, \lor, 1)$ where 
\begin{enumerate}
	\item $(A, \cdot, 1)$ is a commutative monoid
	\item $(A, \land, \lor)$ is a lattice
	\item $(\cdot, \to)$ form a residuated pair.
 \end{enumerate}
We call a commutative residuated lattice \emph{integral} if $1$ is the top element in the lattice order, and \emph{bounded} if it is integral and there is also an extra element $0$ that is the bottom element of the lattice. 
Commutative residuated lattices are a variety called $\mathsf{CRL}$, and their  version with an extra constant $0$ (not necessarily being the lower bound) is a variety called $\mathsf{FL}_e$. The last name comes from logic, as $\mathsf{FL}_e$-algebras are the equivalent algebraic semantics of the Full Lambek calculus with exchange (for more details on residuated structures and substructural logics see \cite{GJKO}). The subvariety of $\mathsf{FL}_e$ consisting of bounded residuated lattices is called $\mathsf{FL}_{ew}$ (where $w$ stands for \emph{weakening} in the associated logical system).

With respect to their structure theory, all varieties of residuated lattices are \emph{ideal-determined} in the sense of \cite{AU92}, and the role of ideals is played by \emph{congruence filters}, that is, (non-empty) lattice filters closed under products. 
In particular, there is a lattice-isomorphism between the lattice of congruence filters and the congruence lattice. Given a commutative integral residuated lattice $\alg A$, the isomorphism is given by the following two maps, where $\theta$ is any congruence of $\alg A$ and $F$ is any congruence filter of $\alg A$:
$$\theta \mapsto F_\theta = \{a \in A: (a, 1) \in \theta\}, \qquad F \mapsto \theta_F =\{(a, b): a \to b, b \to a \in F\}.$$
 Finitely generated (or, equivalently, compact) congruences are then in bijection with finitely generated congruence filters. 
 Now, given $\alg A$ a commutative integral residuated lattice, and $X \sse A$,  the congruence filter generated by $X$ in $\alg A$ is the set $$\Fil_{\alg A}(X)=\{b \in A: a_1 \cdot \ldots \cdot a_n \leq b,\, a_1, \ldots, a_n \in X,\, n \in \mathbb{N}\}.$$
It then follows that finitely generated congruence filters correspond to principal congruence filters (it suffices to take the meet of the generators), and analogously finitely generated congruences are principally generated.

Wajsberg hoops fit in the residuated lattice framework since they are term equivalent to the variety $\WH$ of commutative integral residuated lattices satisfying (T), and the divisibility condition:
$$x \land y = x \cdot (x \to y) \quad (div)$$
which ensure that the order is natural. 
From now on, we shall see a Wajsberg hoop as an algebra in $\WH$, with the signature of residuated lattices $\{\cdot, \to, \land, \lor, 1\}$.

Bounded Wajsberg hoops are term-equivalent to MV-algebras, the equivalent algebraic semantics of infinite-valued \luk\ logic. 
%in the signature $\{\cdot, \to, \land, \lor, 0, 1\}$. Wajsberg algebras are term-equivalent to MV-algebras: an 
MV-algebras are usually introduced (see \cite{CDM}) in the signature $\{\oplus, \neg, 0\}$, and can be defined as algebras $\alg A = (A, \oplus, \neg, 0)$ such that 
\begin{enumerate}
	\item $\neg\neg x = x$,
	\item $x \oplus (\neg 0) = \neg 0$,
	\item $\neg(\neg x \oplus y) \oplus y = \neg(\neg y \oplus x ) \oplus x$.
\end{enumerate} 
As mentioned in the introduction, both Wajsberg hoops and MV-algebras are closely related to lattice-ordered abelian groups, that we shall call \emph{abelian $\ell$-groups} for short. An abelian $\ell$-group is an algebra $\alg G = (G, +, -, 0, \land, \lor )$ where $(G, +, -, 0)$ is an abelian group, $(G, \land, \lor)$ is a lattice, and the group operation distributes over the lattice operations. 

Let $\alg G$ be an abelian $\ell$-group with a \emph{strong unit} $u$, i.e., an element $u$ such that for all $x \in G$, there exists $n \in \mathbb{N}$: $x \leq u + \ldots + u$ ($n$ times). Then Mundici's $\Gamma$ functor constructs an MV-algebra $\Gamma (G, u) = ([0,u], \oplus_u, \neg_u,0)$ defined by:
$[0,u]=\{x \in G: 0\leq x \leq u\}, a \oplus_u b = (a + b) \land u, \neg_u\, a = u - a$. Every MV-algebra can be obtained in this way \cite{M86}. 

Both MV-algebras and abelian $\ell$-groups can be seen in the framework of residuated lattices. The latter are indeed term-equivalent to commutative residuated lattices that satisfy $x\cdot (x\to 1) = 1$, via the translation $x \cdot y := x + y$, $1:= 0$, and $x \to y := -x + y$.

While MV-algebras, as defined above, are term equivalent to the subvariety $\MV$ of $\mathsf{FL}_{ew}$ axiomatized by $(x \to y) \to y = x \lor y$, via the stipulations: $$1 := \neg 0,\,0 := 0,\, x \cdot y := \neg(\neg x \oplus \neg y),\, x \to y := \neg x \oplus y,\, x \lor y = (x \to y) \to y,\, x \land y = x \cdot (x\to y)$$ and $$0 := 0,\; \neg x = x \to 0,\;\;x \oplus y = \neg(\neg x \cdot \neg y).$$ Thus, as it often happens in the literature, in what follows we refer to MV-algebras as algebras in $\MV$, in the signature $\{\cdot, \to, \land, \lor, 0, 1\}$.

Given any abelian $\ell$-group $\alg G$ seen as a residuated lattice, we define its \emph{negative cone} to be the algebra $\alg G^- = (G^-, \cdot, \to_1, \land, \lor, 1)$ on the subset of negative elements $G^- = \{x \in G: x \leq 1\}$ with $x \to_1 y = (x \to y) \land 1$. Negative cones of abelian $\ell$-groups are a variety of residuated lattices \cite{BCGJT}, term equivalent to \emph{cancellative hoops}, that is, hoops where the product satisfies the usual cancellativity law. Cancellative hoops are a subvariety of Wajsberg hoops, and moreover, a totally ordered Wajsberg hoop is either bounded (thus, equivalent to an MV-algebra) or cancellative \cite[Proposition 2.1]{AglianoPanti1999}.

Let us consider the varieties $\WH$ and $\MV$, respectively the varieties of Wajsberg hoops and MV-algebras. Both varieties are generated by an algebra defined over the real unit interval $[0,1]$, called \emph{standard}. More precisely, $[0,1]_{\MV} = ([0,1], \cdot_{\L}, \to_{\L}, \min, \max, 0, 1)$, with $\cdot_{\L}$ and $\to_{\L}$ being respectively \luk\ t-norm and residuum: $$x \cdot_{\L} y = \max\{0,x+y-1\}, \;x \to_{\L} y = \min\{1,y-x +1\}.$$ 
The standard Wajsberg hoop $[0,1]_{\WH}$ is the $0$-free reduct of $[0,1]_{\MV}$.

Since the algebras on $[0,1]$ are generic for $\MV$ and $\WH$, i.e. they generate the two varieties, by a standard universal algebraic argument (see \cite[Lemma 4.98]{MMTbook}) the $n$-generated free algebras are algebras of functions from $[0,1]^{n}$ to $[0,1]$ generated by the projections and with operations defined componentwise by the ones on the standard algebra.
We point out that the understanding of finitely generated free algebras in this context is meaningful from the logical point of view, since the $n$-generated free algebra is isomorphic to the Lindenbaum algebra of $n$-variable formulas of the corresponding logic.
 
By the well-known McNaughton's Theorem, the free $n$-generated MV-algebra is shown to be (isomorphic to) the algebra of McNaughton functions, that is, piecewise linear continuous functions with integer coefficients \cite{Mc51,M94}. Thus, to each McNaughton function corresponds an equivalence class of MV-terms in the appropriate free algebra (equivalently, logically equivalent \luk\ formulas). From now on, we will see the free $n$-generated MV-algebra as its representative $\freeMV(n)$ given by the McNaughton functions over $[0,1]^n$.

In \cite{AglianoPanti1999} the authors show that the free Wajsberg hoop over $n$ generators $\freeW(n)$ can be characterized in terms of the free MV-algebra $\freeMV(n)$ in the following way. Let $\uno$ be the point in the $n$-cube whose coordinates are all $1$, then:
\begin{equation}\label{eq:freeWH}\freeW(n) = \{f \in \freeMV(n) : f(\uno) = 1\}.\end{equation}
In what follows we will refer to the McNaughton functions in $\freeW(n)$ as \emph{Wajsberg functions}.

A key character of the rest of this paper is going to be played by finitely presented Wajsberg hoops. 
A \emph{finitely presented algebra} is an algebra with a finite number of generators, satisfying finitely many identities. In a variety $\vv V$, a finitely presented algebra is then (isomorphic to) a
quotient of a finitely generated free algebra modulo a finitely generated congruence. 
In what follows, we will denote with $\free_{\vv V}(n)/f$ a finitely presented algebra that is the quotient of $\free_{\vv V}(n)$ with respect to a principal congruence filter generated by an element $f \in \free_{\vv V}(n)$.
The core of the present manuscript is a representation of finitely presented Wajsberg hoops in terms of geometrical objects. This connection stems from the following observation, that allows to connect principal filters with polyhedra in unit cubes. 
Given any function $f \in \freeMV(n)$, let $O_{f}$ be its one-set, i.e. 
\begin{equation}\label{eq:oneset} O_{f} = \{ \x \in [0,1]^{n}: f(\x) = 1\}. \end{equation}
Similarly, given an MV-term $t$, we will write $O_t$ for the one-set of its associated McNaughton function.
Then a particular instance of Lemma 4.1 in \cite{AglianoPanti1999} yields the following.
\begin{lemma}\label{lemma:fil-oneset}
Let $f, g \in \freeW(n)$. Then $g \in \Fil_{\freeW(n)}(f)$ if and only if $O_{f} \sse O_{g}$.
\end{lemma}
We shall see in what follows that one-sets of Wajsberg functions correspond to particular \emph{rational polyhedra}. By a  {\em (rational) convex polyhedron} (or {\em (rational) polytope}) of $\mathbb{R}^k$ we mean the convex hull of finitely many points of $\mathbb{R}^k$ ($\mathbb{Q}^k$ respectively), and a (rational) polyhedron is  a finite union of (rational) polytopes.

\section{Positive reducts}
As it follows from (\ref{eq:freeWH}), the free Wajsberg hoop on $n$ generators $\freeW(n)$ is (isomorphic to) a particular subreduct of the free MV-algebra over the same number of generators. We shall see how this arises from both general considerations and properties more specific to the MV-algebraic-theory. 

Let us first consider a variety $\vv V$ of $\mathsf{FL}_e$-algebras and the variety $\vv V_0$ of their $0$-free subreducts. This last sentence makes sense because of the following (folklore) fact, which we prove here for the sake of the reader, since we did not find an appropriate reference. 
\begin{proposition}
Let $\vv V$ be a variety of $\mathsf{FL}_e$-algebras, then the class $\vv V_0$ of its $0$-free subreducts is a variety of commutative residuated lattices.
\end{proposition}
\begin{proof}
	We show that $\vv V_0$ is closed under homomorphic images, subalgebras, and direct products. 
	
	Suppose first that $\alg S \in \vv V_0$ is a subreduct of some $\alg A \in \vv V$, and that $\alg R = h(\alg S)$ for a homomorphism $h$ of residuated lattices. Let us call $\alg A_0$ the $0$-free reduct of $\alg A$, then $\alg A_0 \in \vv V_0$ and $\alg S$ is a subalgebra of $\alg A_0$, in symbols $\alg S \leq \alg A_0$. Since commutative residuated lattices have the congruence extension property (\cite[Lemma 3.57]{GJKO}, \cite{GalPhd}), if we set $\theta = \ker(h)$, there exists a congruence $\theta_0$ of $\alg A_0$ such that $\theta_0 \cap (S \times S) = \theta$. Then it follows from the Third Isomorphism Theorem (\cite[Theorem 6.18]{BS}) that $\alg S/\theta \leq \alg A_0 /\theta_0 $. Finally, since $\alg A_0$ is bounded so is $\alg A_0/\theta_0$, which then is the $0$-free reduct of a quotient of $\alg A$. Hence, $\alg R = h(\alg S)$ is a subreduct of an algebra in $\vv V$, and $\vv V_0$ is closed under homomorphic images.
	
	It is easy to see that $\vv V_0$ is closed under subalgebras. Indeed, if $\alg R \leq \alg S$ and $\alg S$ is a residuated lattice that is the subreduct of some $\alg A \in \vv V$, it is clear that $\alg R$ is a subreduct of $\alg A$ as well.
	
	Lastly, we show that $\vv V_0$ is closed under direct products. Let $\{\alg S_{i}\}_{i \in I}$ be a family of subreducts of algebras in $\vv V$, more precisely say that $\alg S_i \leq \alg A_{i,0}$, where each $\alg A_{i, 0}$ is the $0$-free reduct of some $\alg A_i \in \vv V$. Then, denoting with $P$ and $S$ the algebraic operators denoting the class of direct products and subalgebras respectively,
	 $$\prod_{i \in I} \alg S_i \in PS(\{\alg A_{i,0}\}_{i \in I})\subseteq SP(\{\alg A_{i,0}\}_{i \in I}).$$
	Thus, it follows that $\prod_{i \in I} \alg S_i$ is a subalgebra of the $0$-free reduct of $\prod_{i \in I} \alg A_i$, completing our proof.
\end{proof}
It is worth pointing out that the previous proof works in the very same way for not necessarily commutative varieties of bounded residuated lattices that do have the congruence extension property. Notice also that, in general, the class of $0$-free subreducts of a variety of $\mathsf{FL}$-algebras is a quasivariety. In fact, the closure under subalgebras and direct products requires no special properties, and it can be shown that the subreducts are closed under ultraproducts. The fact that the class of subreducts of a quasivariety is a quasivariety has been shown to hold in general by Mal'cev \cite{Mal71}.

We will write $\mathbf{T}_{\vv V_{0}}(X)$ for the term algebra over $X$ in the language of $\vv V_{0}$, and $\mathbf{T}_{\vv V}(X)$ for the term algebra over the same set of generators $X$ in the language of $\vv V$. Given any variety $\vv W$ and any set of variables $X = \{x_1, \ldots, x_n\}$, an \emph{assignment} of $X$ into an algebra $\alg A \in \vv W$ is a function $h$ mapping each variable $x_i$ to an element of $\alg A$, say $h(x_i) = a_i \in A$, for $i = 1, \ldots, n$. Then $h$ extends to a homomorphism (that we will also call $h$) from the term algebra $\mathbf{T}_{\vv W}(X)$ to $\alg A$. Moreover, given terms $t(x_1, \ldots, x_n), u(x_1, \ldots, x_n)$ over the language of $\vv W$, we write $\alg A \models t \approx u$ if for any assignment $h$ of the variables $x_1, \ldots, x_n$ to elements of $\alg A$, $h(t(x_1, \ldots, x_n)) = h(u(x_1, \ldots, x_n))$ in $\alg A$. We write $\vv W \models t \approx u$ if $\alg A \models t \approx u$ for all $\alg A \in \vv W$.
\begin{proposition}\label{prop:freereduct}
Let $\vv V$ be a variety of $\mathsf{FL}_{ew}$-algebras and $\vv V_{0}$ be the variety of its $0$-free subreducts. Let $X$ be any set, then $\free_{\vv V_{0}}(X)$ is isomorphic to the subalgebra of the $0$-free reduct of $\free_{\vv V}(X)$ generated by $X$.
\end{proposition}
\begin{proof}
Clearly $\mathcal{T}_{\vv V_{0}}(X)$ is exactly the subalgebra of the $0$-free reduct of $\mathcal{T}_{\vv V}(X)$ generated by $X$. 
Thus, the thesis follows from the fact that, given terms $a, b \in \mathcal{T}_{\vv V_{0}}(X) \sse \mathcal{T}_{\vv V}(X)$, $\vv V_{0} \models a \approx b$ iff $\vv V \models a \approx b$.
Indeed, if $\vv V_0 \not\models a \approx b$, there exists $\alg A_0$ subreduct of $\alg A \in \vv V$ that does not satisfy $a \approx b$ for some assignments to the variables in $a, b$ to elements of $\alg A_0$. The same elements seen in $\alg A$ testify that also $\vv V \not\models a \approx b$.
  Vice versa, if $\vv V \not\models a \approx b$, for instance $\alg A \not\models a \approx b$, also the $0$-free reduct of $\alg A$ does not satisfy the identity, which means that $\vv V_{0} \not\models a \approx b$.
\end{proof}

We now focus on the case where $\mathsf{V} = \MV$ and $\mathsf{V}_{0} = \WH$, where the above proposition can be refined.
We call an MV-term \emph{positive} if it is written in the $0$-free signature $\{\cdot, \to, \land, \lor, 1\}$. We call an MV-term \emph{negative} if it is the negation of a positive term, where as usual $\neg x := x \to 0$.
The following result also follows from \cite[Lemma 1]{AB10}.
\begin{lemma}\label{prop:posneg}
Every MV-term $t$ is either equivalent in $\MV$ to a positive term or to a negative term.
\end{lemma} 
\begin{proof}

We prove this by induction on the construction of the term $t$. 
For the base case, obviously the constant 1 and the generators are positive terms, while $\MV \models 0 \approx \neg 1$. 

Let now $f, g$ be positive terms, we show that $f * g, f * \neg g, \neg f * g, \neg f * \neg g$ for $* \in \{\cdot, \to\}$ are either equivalent to positive terms themselves or their negation is, which implies the claim.

Let us start with $* = \cdot$. Clearly $f \cdot g$ is positive. The negation of $f \cdot \neg g$ is equivalent to a positive term, indeed in $\MV$, $\neg (f \cdot \neg g) = f \to \neg\neg g = f \to g$. For the last case, $\neg(\neg f \cdot \neg g) = (f \to (f \cdot g)) \to g$ (see \cite{AglianoPanti1999}, but it can also be easily checked in the standard MV-algebra), which is positive. 

Now we check $* = \to$. Clearly $f \to g$ is positive, and $\neg f \to \neg g = \neg (\neg f \cdot g) = g \to f$. Moreover, $f \to \neg g = \neg (f \cdot g)$, thus its negation is equivalent to a positive term. Finally, $\neg f \to g = \neg(\neg f \cdot \neg g) =  (f \to (f \cdot g)) \to g$ which is positive.

Notice that a term cannot be equivalent to both a positive and a negative MV-term. Indeed, if we evaluate on any MV-algebra all variables to $1$, such an evaluation of a positive term gives $1$, while it would give $0$ on a negative term. 
\end{proof}
Notice that in the proof of the lemma we used both involutivity and the identity $\neg(\neg f \cdot \neg g) = (f \to (f \cdot g)) \to g$ that does not hold in all involutive structures, thus the lemma cannot be generalized outside of $\mathsf{MV}$ in a straightforward way.

Let us then call a function $f \in \free_{\sf MV}(n)$ \emph{positive} if it corresponds to the equivalence class of a positive term.
Then the previous lemma also gives us a different point of view on the description of the finitely generated free Wajsberg hoops in \cite{AglianoPanti1999}, indeed we can show that positive McNaughton functions are exactly those that are $1$ in ${\bf 1}$.
\begin{proposition}\label{prop:pos}
A function $f \in \free_{\sf MV}(n)$ is positive if and only if $f({\bf 1}) = 1$.
\end{proposition}
\begin{proof}
If a function $f$ is positive, then the statement can be easily seen by induction on the construction of a positive term $f$ corresponds to, given that the operations are defined pointwise. 
Vice versa, since positive terms are $1$ in $\bf 1$, negative terms are clearly $0$ and the claim follows from Proposition \ref{prop:posneg}. 
\end{proof}
Thus we get an alternative proof of Theorem 3.1 in \cite{AglianoPanti1999}, as a consequence of Propositions $\ref{prop:freereduct}$ and \ref{prop:pos}. A similar line of thought to show Aglian\`o and Panti's result has been followed in \cite{ACV}, however what we presented here is in higher generality and contains more details, as they are relevant for what follows.
\begin{corollary}
The free Wajsberg hoop over $n$ generators is isomorphic to $\freeW(n) = \{f \in \freeMV(n) : f(\uno) = 1\}$.
\end{corollary}

We will see that the connection described in the last corollary lifts to a categorical equivalence between finitely presented Wajsberg hoops and a subcategory of finitely presented MV-algebras, given by quotients of free MV-algebras with respect to positive McNaughton functions.

As shown in \cite{Mundici11}, finitely presented MV-algebras are isomorphic to algebras of McNaughton functions restricted to rational polyhedra.
In particular, given $\scP \sse [0,1]^{n}$, we write $\M(\scP)$, for the algebra of functions in $\freeMV(n)$ restricted to $\scP$:
\begin{equation}\label{eq:MP}
\M(\scP) = \{f_{\rest \scP} : f \in \freeMV(n)\}
\end{equation}
Moreover, a finitely presented MV-algebra of the kind $\freeMV(n)/f$ is isomorphic to $\M(O_{f})$ \cite{Mundici11}. 
We shall then now see how one can characterize one-sets of positive McNaughton functions. 
\begin{definition}
For any $n \in \mathbb{N}$, we call a polyhedron in $[0,1]^{n}$ \emph{pointed} if it contains the vertex $\uno$. 
\end{definition}
\begin{lemma}\label{thm:unirational}
Let $\scP \subseteq [0,1]^{n}$. The following are equivalent.
\begin{enumerate}
\item $\scP$ is a pointed rational polyhedron.
\item $\scP = O_{f}$ for a positive function $f \in \freeMV(n)$.
\end{enumerate}
\end{lemma}
\begin{proof}
The proof follows from \cite[Corollary 2.10]{Mundici11}, stating that every rational polyhedron in $[0,1]^n$ is the $1$-set of a McNaughton function $f \in \freeMV(n)$, and Proposition \ref{prop:pos}, that shows that a McNaughton function $f$ is positive if and only if $f(\uno) = 1$.
\end{proof}

\begin{lemma}\label{lemma:posquot}
Consider a pointed rational polyhedron $\scP \sse [0,1]^n$, then if $g_{\rest \scP} = h_{\rest \scP}$ in $\M(\scP)$, $g$ is positive if and only if $h$ is positive.
\end{lemma}
\begin{proof}
Since $\scP$ is pointed, it contains the point $\one$.
Thus if two functions $g$ and $h$ in $\freeMV(n)$ coincide over $\scP$, they have the same value at $\one$ and the claim follows. 
\end{proof}

The previous lemma justifies the following definition. For all unexplained notions of category theory we refer the reader to \cite{Mac}.
\begin{definition}
We call \emph{positive} an MV-algebra of the kind $\M(\scP)$, where $\scP$ is a pointed rational polyhedron.
Let $\scP \sse[0,1]^n, \scQ \sse [0,1]^m$ be pointed rational polyhedra. A homomorphism of positive MV-algebras $\alpha: \M(\scP) \to \M(\scQ)$ is \emph{positive} if for all positive functions $h \in \free_{\MV}(n)$, $\alpha(h_{\rest\scP})$ is mapped to the restriction of a positive function $k$, $\alpha(h_{\rest\scP}) = k_{\rest \scQ}$.

We call $\MV^{+}$ the category of positive MV-algebras with positive homomorphisms. 
\end{definition}
Notice that $\MV^{+}$ is a category since the identity maps are positive homomorphisms and a composition of positive homomorphisms is still positive. Moreover, $\MV^{+}$ is a (non-full) subcategory of finitely presented MV-algebras.

Let now $n \in \nn$, and $\scP \sse [0,1]^{n}$ be a pointed rational polyhedron. We define the \emph{Wajsberg hoop relative to $\scP$}, $\W(\scP)$, to be the algebra of functions in $\freeW(n)$ restricted to $\scP$:
\begin{equation}\label{eq:WP}
\W(\scP) = \{f_{\rest \scP} : f \in \freeW(n)\}
\end{equation}
Notice that this is a well-defined Wajsberg hoop since the operations of the free algebra are defined pointwise.
We show that, analogously to the MV-algebraic case, all finitely presented Wajsberg hoops are of this kind (up to isomorphism).
\begin{lemma}\label{lemma:isoquotient}
Given $f \in \freeW(n)$, the map $\omega: \W(O_{f}) \to \freeW(n)/f$ defined as $\omega(g_{\rest O_{f}}) = g/f$ is an isomorphism.
\end{lemma}
\begin{proof}
We first show that the map $\omega$ is well-defined, that is, if $g, h \in \freeW(n)$, and $g(\x) = h(\x)$ for $\x \in O_{f}$ then $g/f = h/f$. By the correspondence between congruences and filter, $g/f = h/f$ if and only if $g \to h, h \to g \in \Fil(f)$, the congruence filter generated by $f$ in $\freeW(n)$. This, by Lemma \ref{lemma:fil-oneset}, happens if and only if $O_{f} \sse O_{g \to h}, O_{h \to g}$, that is, for all $\x \in O_{f}, (g \to h)(\x) = 1 $ and $(h \to g) (\x) = 1$. This holds, since $g$ and $h$ coincide on $O_{f}$.

It is easy to see that $\omega$ is an homomorphism. Indeed, clearly $\omega(1_{\rest O_f}) = 1/f$. Moreover, given $g, h \in \freeW(n)$ and $* \in \{\cdot, \to\}$, since operations are defined pointwise, $\omega(g_{\rest O_{f}} * h_{\rest O_{f}}) = \omega((g* h)_{\rest O_{f}}) = (g*h)/f = g/f * h/f = \omega(g_{\rest O_{f}}) * \omega(h_{\rest O_{f}})$.

We now prove that $\omega$ is injective. Suppose $g_{\rest O_{f}} \neq h_{\rest O_{f}}$, then there is $\x \in O_{f}$ such that $g(\x) \neq h(\x)$. Thus either $g(\x) \to h(\x) <1 $ or $h(\x) \to g(\x) <1$, which implies that either $O_{f} \not\sse O_{g \to h}$ or $O_{f} \not\sse O_{h \to g}$, thus either $g \to h$ or $h \to g \notin \Fil(f)$. We conclude that $g/f \neq h/f$, and thus $\omega$ is injective. Surjectivity is easily checked, since given $g/f \in \freeW(n)/f$, $\omega(g_{\rest O_{f}}) = g/f$.
\end{proof}
Therefore we obtain the following.
\begin{theorem}\label{thm:finpresW}
Let $\alg A$ be a Wajsberg hoop, then the following are equivalent.
\begin{enumerate}
\item $\alg A$ is finitely presented.
\item $\alg A \cong \W(\scP)$ for a pointed polyhedron $\scP \sse [0,1]^{n}$ for some $n \in \mathbb{N}$.
\end{enumerate}
\end{theorem}
\begin{proof}
Directly follows from Lemma \ref{lemma:isoquotient} and Lemma \ref{thm:unirational}. \end{proof}

Let $\WH_{\mathsf{fp}}$ be the category of finitely presented Wajsberg hoops with homomorphisms, that is a full subcategory of the algebraic category of Wajsberg hoops.  
We now define a map $\Phi: \MV^{+} \to \WH_{\mathsf{fp}}$. For all pointed rational polyhedra $\scP \sse [0,1]^n, \scQ \sse [0,1]^m$, and positive homomorphism $\alpha: \M(\scP) \to \M(\scQ)$
\begin{equation}\label{eq:functphi}
\begin{split}
\Phi(\M(\scP)) &= \W(\scP) \\
\Phi(\alpha): \W(\scP) \to \W(\scQ), \;\Phi(\alpha) (f_{\rest \scP}) &= \alpha(f_{\rest \scP}) .
\end{split}
\end{equation}
\begin{theorem}\label{thmcateq}
Finitely presented Wajsberg hoops with homomorphisms and positive finitely presented MV-algebras with positive homomorphisms are categorically equivalent.
\end{theorem}
\begin{proof}
The fact that $\Phi$ is a functor from $\MV^{+}$ to $\WH_{\mathsf{fp}}$ is easy to check and essentially follows from Lemmas \ref{thm:unirational}, \ref{lemma:posquot} and \ref{lemma:isoquotient}.

It suffices to show that $\Phi$ is full, faithful and essentially surjective (\cite[Theorem 1, p. 91]{Mac}). Let us first show the latter. Let $\alg A$ be a finitely presented Wajsberg hoop, then by what we recalled in the preliminaries in Section \ref{sec:prel}, $\alg A$ is isomorphic to a quotient of the kind $\freeW(n)/f$, which by Lemma \ref{lemma:isoquotient}, is isomorphic to $\W(O_f)$. Since $f$ is a positive function, $O_f$ is a pointed rational polyhedron by Lemma \ref{thm:unirational}.
Since $\Phi(\M(O_f)) = \W(O_f)$, $\Phi$ is essentially surjective.

We now show that $\Phi$ is faithful. Consider two positive homomorphisms $\alpha_1, \alpha_2$ from $\M(\scP)$ to $\M(\scQ)$, that differ on some $f_{\rest \scP} \in \M(\scP)$. Since either $f$ is positive or $\neg f$ is, it follows from the definition and involutivity ($\neg\neg x = x$ in $\MV$) that $\Phi(\alpha_1)$ and $\Phi(\alpha_2)$ either differ on $f_{\rest \scP}$ or on $\neg f_{\rest \scP}$. Thus $\Phi$ is faithful. 

It is left to prove that the functor is full. Let $\beta: \W(\scP)  \to \W(\scQ)$ be a homomorphism of Wajsberg hoops. Consider the homomorphism $\beta^+: \M(\scP)  \to \M(\scQ)$ defined as follows. Suppose $\beta(f_{\rest \scP}) = g_{\rest \scQ}$, then define $\beta^+(f_{\rest \scP}) = g_{\rest \scQ}$ and $\beta^+(\neg f_{\rest \scP}) = \neg g_{\rest \scQ}$. $\beta^+$ is well defined since given $f, h \in \freeMV(n)$ that coincide over $\scP$, they are either both positive or both negative (Lemma \ref{lemma:posquot}). We show that $\beta^+$ is a homomorphism. It suffices to check that it preserves a constant, product and negation. The fact that it preserves constants and negation follows directly from the definition. For the product, we show one of the cases, the other ones being similar. Let $h, k \in \freeW(n)$, and suppose $\beta(h_{\rest \scP}) = x_{\rest \scQ}$ and $\beta(k_{\rest \scP}) = y_{\rest \scQ}$, then
\begin{equation*}
\begin{split}
\beta^+ (h_{\rest \scP}) \cdot \beta^+(\neg k _{\rest \scP}) = x_{\rest \scQ}  \cdot (\neg y_{\rest \scQ}) = \neg (x \to y )_{\rest \scQ} \\ = \beta^+ [\neg (h \to k)_{\rest \scP}] = \beta^+ ((h_{\rest \scP}) \cdot (\neg k_{\rest \scP}))
\end{split}
\end{equation*}
Finally, clearly $\Phi(\beta^+) = \beta$. Thus the functor is full and the proof is complete.
\end{proof}

We point out that there always is an adjunction between an algebraic category and the subcategory of the reducts given by the forgetful functor (see \cite{ACV} for the case of $\WH$ and \cite{Mo18} for a general approach). The connection we highlighted here is meant to provide a clear understanding of the geometrical duality in the following section. 

\section{Duality with pointed rational polyhedra}
In this section we present the duality between finitely presented Wajsberg hoops and pointed rational polyhedra, which follows from the categorical equivalence shown in the previous section, and the restriction to positive MV-algebras of the well-known duality between finitely presented MV-algebras and rational polyhedra in \cite{MarraSpada13}. 

In more details, in \cite{MarraSpada13} the authors show that finitely presented MV-algebras with homomorphisms are dual to a category whose objects are rational polyhedra $\scP \sse \mathbb{R}^{n}$ for some $n \in \mathbb{N}$, and the morphisms are $\mathbb{Z}$-maps between them. We recall that given a rational polyhedron $\scP \sse \mathbb{R}^{n}$ and a continuous map $z = (z_{1}, \ldots, z_{m}): \scP \to \mathbb{R}^{m}$, with $m, n \in \zz^{+}$, $z$ is a $\zz$-map (of $\scP$) if for each $j = 1,\ldots,m$, $z_{j}$ is piecewise linear with integer coefficients.
For details on rational polyhedra and $\mathbb{Z}$-maps the reader is referred to \cite{Mundici11}. 

Since any rational polyhedron in $\mathbb{R}^{n}$ is $\mathbb{Z}$-homeomorphic to a rational polyhedron in $[0,1]^{n}$ (Claim 3.5 in Theorem 3.4 in \cite{MarraSpada13}), the category of rational polyhedra in the duality with finitely presented MV-algebras can be restricted to account only for polyhedra in $[0,1]^{n}$ for some $n \in \mathbb{N}$. 

Here we choose to use this version of the duality since, although less general from the geometrical point of view, it has a clearer algebraic meaning that will be useful in the following sections: the polyhedra are the one-sets of McNaughton functions generating the principal filter associated to a finitely presented algebra. Another choice that we make here is that, given a finitely presented MV-algebra of the kind $\freeMV(n)/f$, we consider the associated polyhedron to be $O_{f}$, the one-set of the function $f$, instead of (equivalently, in the case of MV-algebras, see \cite[Corollary 2.10 ]{Mundici11}) picking the zero-set of the function $f$. 

By Lemma \ref{thm:unirational}, pointed polyhedra are mapped to positive MV-algebras by the Marra-Spada duality in \cite{MarraSpada13}. We will now show what positive homomorphisms correspond to on the geometrical side. In the full duality in \cite{MarraSpada13}, homomorphisms between finitely presented MV-algebras are exactly those obtainable in the following way via the functor $\M$ from the category of rational polyhedra to finitely presented MV-algebras that established the duality. Let $\scP \sse [0,1]^{n}, \scQ \sse [0,1]^{m}$ be rational polyhedra, and $\lambda: \scP \to \scQ$ a $\mathbb{Z}$-map. Then the map dual to $\lambda$, $\M(\lambda)$, is such that: $$f \in \M(\scQ) \mapsto f \circ \lambda \in \M(\scP)$$
and all homomorphisms between $\M(\scQ)$ and $\M(\scP)$ can be obtained in this way. Therefore we obtain the following characterization.
\begin{lemma}\label{lemma:posmorphism}
Given positive functions $f \in \freeMV(n), g \in \freeMV(m)$, the following are equivalent.
\begin{enumerate}
\item $\alpha: \M(O_f) \to \M(O_g)$ is a positive homomorphism of positive finitely presented MV-algebras.
\item $\alpha = \M(\lambda)$, where $\lambda: O_{g} \to O_{f}$ maps $\uno \in O_{g}$ to $\uno \in O_{f}$.
\end{enumerate}
\end{lemma}
\begin{proof}
	 (2) implies (1) by Proposition \ref{prop:pos}. Vice versa, suppose (1) holds. Then by the duality in \cite{MarraSpada13}, every such homomorphism $\alpha$ is of the kind $\M(\lambda)$, for some $\mathbb{Z}$-map $\lambda$. $\lambda $ necessarily maps $\alg 1$ to $\alg 1$. Indeed, if not, $\alpha = \M(\lambda)$ would not map (the restriction of) the identity map, which is positive, to a positive function, and therefore $\alpha$ would not be positive, a contradiction.
\end{proof}
Thus we shall define the following class of $\zz$-maps.
\begin{definition}
Given pointed rational polyhedra $\scP \sse [0,1]^n$ and $\scQ\sse [0, 1]^{m}$, we shall call a map $\zeta:  = (z_{1}, \ldots, z_{m}): \scP \to \scQ$ a \emph{pointed $\zz$-map} if it is the restriction to $\scP$ and co-restriction to $\scQ$ of a $\zz$-map from $[0,1]^n$ to $[0,1]^m$ such that $\zeta(\uno)=\uno$. 
\end{definition}
 Since $\zz$-maps are closed under compositions, so are pointed $\zz$-maps.
Examples of pointed $\zz$-maps are given by the identity function, a function constantly equal to $\one$, the projection of a pointed polyhedron $\scP \sse [0, 1]^{n}$ onto the hyperplane $x^{n} = 1$. Moreover, every Wajsberg function $f \in \freeW(n)$ is a pointed $\zz$-map of $[0, 1]^{n}$ into $[0, 1]$.

\begin{lemma}\label{lemma:pzmap}
Let $\scP \sse [0,1]^{n}$ be a pointed rational polyhedron and consider a continuous function $\zeta = (z_{1}, \ldots, z_{m}): \scP \to [0,1]^{m}$.  $\zeta$ is a pointed $\zz$-map if and only if there are Wajsberg functions $f_{1},\ldots, f_{m} \in \freeW(n)$ such that $\zeta =  (f_{1},\ldots, f_{m})_{\rest \scP}$.
\end{lemma}
\begin{proof}
The statement follows from \cite[Proposition 3.2]{Mundici11}. Indeed, $\zeta$ is a pointed $\zz$-map iff $\zeta = (f_{1},\ldots, f_{m})_{\rest \scP}$ for some McNaughton functions $f_{1},\ldots, f_{m} \in \freeMV(n)$ such that  $f_{i}(\uno) = 1$ for $i = 1,\ldots m$, iff $\zeta = (f_{1},\ldots, f_{m})_{\rest \scP}$ for some Wajsberg functions $f_{1},\ldots, f_{m} \in \freeW(n)$.
\end{proof}
Let $\mathsf{P^+_{\zz}}$ be the category of pointed rational polyhedra with pointed $\zz$-maps. Let $\M$ be the previously described functor in the duality for MV-algebras, from rational polyhedra to finitely presented MV-algebras. Composing the restriction of $\mathcal{M}$ to the category of $\mathsf{P^+_{\zz}}$ with the functor $\Phi$ (notice that the composition is well defined thanks to Lemmas \ref{lemma:posmorphism} and \ref{thm:unirational}), we obtain a functor $\mathcal{W}$ from $\mathsf{P^+_{\zz}}$ to the category of finitely presented Wajsberg hoops. In particular, given pointed rational polyhedra $\scP \sse[0,1]^n$ and $\scQ \sse[0,1]^m$, and $\zeta: \scP \to \scQ$ a pointed $\Z$-map:
\begin{itemize}
	\item $\W(\scP)$ is the algebra of restricted Wajsberg function as in \ref{eq:WP}.
	\item Given $f \in \W(\scQ)$, $\W(\zeta)(f) = f \circ \zeta \in \W(\scP)$.
\end{itemize} 
We then obtain the following result.
\begin{theorem}[Duality for finitely presented Wajsberg hoops]\label{thm:duality}
The functor $\mathcal{W}$ yields a duality between the category of pointed rational polyhedra $\mathsf{P^+_{\zz}}$ and the category of finitely presented Wajsberg hoops $\WH_{\mathsf{fp}}$.
\end{theorem}
\begin{proof}
It follows from  Lemmas \ref{lemma:posmorphism} and \ref{thm:unirational} that the duality in \cite[Theorem 3.4]{MarraSpada13} restricts to 
a duality between the categories $\mathsf{P^+_{\zz}}$ and positive finitely presented MV-algebras with positive homomorphisms.
The claim then follows by the categorical equivalence in Theorem \ref{thmcateq}.
\end{proof}
The duality we obtained has the very same essence of the one about finitely presented MV-algebras, indeed it could have been shown by adapting the proofs in \cite{MarraSpada13}. Here however, we decided to highlight the connection between finitely presented Wajsberg hoops and the subcategory of positive finitely presented MV-algebras which we believe has interest of its own.

We point out the following description of monomorphisms and epimorphisms in the category of pointed polyhedra, that is, left and right cancellable morphisms respectively.
Following \cite{Ca15}, let us call \emph{strict} a pointed $\Z$-map that is a $\Z$-homeomorphism (a homeomorphism via a $\Z$-map whose inverse is also a $\Z$-map) onto its image.
\begin{proposition}\label{prop:dualmorph}
	Let $\scP, \scQ$ be pointed rational polyhedra, and $\zeta: \scP \to \scQ$ a pointed $\mathbb{Z}$-map.
	\begin{enumerate}
		\item\label{prop:dualmorph1} $\zeta$ is a monomorphism if and only if it is injective.
		\item\label{prop:dualmorph2} $\zeta$ is an epimorphism if and only if it is surjective.
		\item\label{prop:dualmorph3} $\zeta$ is strict if and only if $\W(\zeta)$ is a surjective homomorphism from $\W(Q)$ to $\W(P)$.
	\end{enumerate}
\end{proposition}
\begin{proof}
(1) and (2) can be shown analogously to \cite[Theorems 3.2]{Ca15}. 

%we point out the necessary modifications. Recall that in concrete categories injective morphisms are monomorphisms, and surjective morphisms are epimorphisms.
%For (1), we suppose $\zeta$ is a monomorphism but it is not injective. Then we derive a contradiction by definining $\mu_1, \mu_2$ such that $\zeta \circ \mu_1 = \zeta \circ \mu_2$ but $\mu_1 \neq \mu_2$ in the following way. Let $x,y \in \scP$ be rational points such that $x \neq y$, $\zeta(x) = \zeta(y)$. Call ${\rm den}(v)$ the least common denominator of the coordinates of $v$, and let $d = 1 /({\rm den}(x) \cdot {\rm den}(y))$ if ${\rm den}(x) \cdot {\rm den}(y) \neq 1$, and let $d = 0$ otherwise. Let $\mathscr{S} =\{d,1\}$. Then arguing in the same way as in the proof of \cite[Theorems 3.2]{Ca15}(i), we obtain that the maps $\mu_1, \mu_2: \mathscr{S} \to \scP$ defined by $\mu_1(d) = x, \mu_2(d) = y, \mu_1(1)= \mu_2(1) = \one$ are (pointed) $\zz$-maps, and they contradict the fact that $\zeta$ is a monomorphism since $\zeta \circ \mu_1 = \zeta \circ \mu_2$ but $\mu_1 \neq \mu_2$.
%
%For (2), suppose $\zeta$ is an epimorphism but it is not onto $\scQ$. Then we can derive a contradiction by definining $\mu_1, \mu_2: \scQ \to \scQ \times [0,1]$ such that $\mu_1 \circ \zeta = \mu_2 \circ \zeta$ but $\mu_1 \neq \mu_2$ in the following way: $\mu_1(x) = (x, 1)$, and $\mu_2(x) = (x,1)$ if $x \in \zeta(\scP)$, and $\mu_2(x)= (x, 0)$ otherwise. Then, arguing as in the proof of \cite[Theorems 3.2]{Ca15}(ii), $\mu_1$ and $\mu_2$ are (pointed) $\zz$-maps, and they contradict the fact that $\zeta$ is an epimorphism.

The proof of (3) follows from Theorem \ref{thmcateq} and \cite[Theorem 3.5]{Ca15}.
\end{proof}

Let us end this section with some relevant examples of finitely presented Wajsberg hoops and their associated polyhedra. Let us call $\alg W_{n}$ the Wajsberg chain with $n$-elements.
\begin{example}\label{prop:wn}
It is easy to see that
$\alg W_{1}$ is isomorphic to $\W(\{1\})$, $\alg W_{2}$ is isomorphic to $\W(\{0, 1\})$, and for all $n \geq 3$, $\alg W_{n}$ is isomorphic to $\W(\{\frac{1}{n-1}, 1\})$ .

Let us first consider $\W(\{1\})$. Since all Wajsberg functions have value $1$ at $\one$, $\freeW(1)$ restricted to $\{1\}$ is the trivial algebra with one element $\alg W_1$.
If we now consider $\alg W_{2}$, observe that Wajsberg functions are either $0$ or $1$ at $0$, and can only have value $1$ in $1$. Thus, the free algebra of Wajsberg functions with one variable restricted to $\{0,1\}$ is isomorphic to the $0$-free reduct of the $2$-element Boolean algebra $\alg W_{2}$. 

The cases of $\W(\{\frac{1}{n-1}, 1\})$, for $n \geq 3$, can be shown by induction on the construction of the elements of the free algebra. $\freeW(1)$ is indeed generated by the projection function $\pi_{1}$ (that is the identity map), with the operations defined pointwise. Thus, the restriction of the 1-generated Wajsberg functions to $\{\frac{1}{n-1}, 1\}$ is a Wajsberg subhoop of $[0,1]$ generated by $\frac{1}{n-1}$ and $1$, which is isomorphic to $\alg W_{n}$.

Notice that in MV-algebras, on the other hand, the finite MV-chain of $n$ elements is isomorphic to $\M(\{\frac{1}{n-1}\})$.
\end{example}
It is well known (see \cite{AFM}) that the algebras $\alg W_n$ from the previous example are, up to isomorphisms, all the finite subdirectly irreducible Wajsberg hoops. 
In the next proposition we show that actually, they are exactly all of the finitely presented subdirectly irreducible Wajsberg hoops.
%We shall show that they are also the only subdirectly irreducible Wajsberg hoops that are finitely presented. 
\begin{proposition}\label{prop:subdire}
Finitely presented subdirectly irreducible Wajsberg hoops are exactly 
the algebras $\alg W_n$, for $n \geq 1$.
\end{proposition}
\begin{proof}
First, it is well known that the algebras $\alg W_n$ are finite (hence finitely presented) subdirectly irreducible Wajsberg hoops, see for instance \cite{AFM}.

For the converse direction, we first recall that subdirectly irreducible Wajsberg hoops are totally ordered. Indeed, the prelinearity equation $(x \to y) \lor (y \to x) \approx 1$ holds in $\WH$ (see \cite{AglianoPanti1999}), which characterizes varieties of commutative and integral residuated lattices whose subdirectly irreducible algebras are totally ordered (see \cite{GJKO}). 
Now, by Theorem \ref{thm:finpresW}, a finitely presented Wajsberg hoop is an algebra isomorphic to $\W(\scP)$, for some $n \in \mathbb{N}$ and some pointed rational polyhedron $\scP \sse[0,1]^n$. Suppose that $\scP$ contains two different points $\x \neq {\bf y}$ of $[0,1]^{n}$, with $\x, {\bf y} \neq \one$. Then one can construct two piecewise linear functions with integer coefficients that map $\one$ to 1, but are incomparable with respect to the pointwise order when restricted to $\scP$. 
Such McNaughton functions can be constructed using Schauder's hats, we omit the technical details here (for details on Schauder's hats the reader can check \cite{CDM}).
Thus, if $\scP$ contains two points different from each other and from $\uno$, $\W(\scP)$ is not totally ordered, and thus it is not subdirectly irreducible. 

Thus necessarily $\scP$ has at most a pair of points $\{\x, \uno\}$, where $\x = (x_{1}, \ldots, x_{n}) \in [0,1]^{n} \cap \mathbb{Q}^n$ (and possibly $\x = \uno$). Since Wajsberg functions are necessarily $1$ at $\uno$, $\W(\scP)$ is (isomorphic to) the subalgebra of $[0,1]$ generated by the coordinates $x_{1}, \ldots, x_{n}$. If $\x = \uno$ then $\W(\scP) \cong \alg W_1$, and if $\x = \alg 0$ then $\W(\scP) \cong \alg W_2$.
Otherwise, writing each $x_i$ as an irreducible fraction $p_i / q_i$, and letting $m = {\rm lcm}(q_1, \ldots, q_n)$, it is easy to see that then $\W(\scP)$ is the finite chain generated by $1 / m$, $\alg W_{m+1}$ for $m \geq 2$. Notice that indeed bounded subalgebras of $[0,1]$ are either dense or isomorphic to some $\alg W_k$, $k \in \mathbb{N}$, see \cite[Proposition 3.5.3]{CDM}. This completes the proof.
\end{proof}
Notice that the analogous proof for MV-algebras shows that the finite MV-chains are exactly all of the finitely presented subdirectly irreducible MV-algebras.

Aside from the algebras $\alg W_n$, other relevant subdirectly irreducible Wajsberg hoops can be defined as follows. 
Let $\bf Z$ be the abelian $\ell$-group of the integers naturally ordered, then
we call $\bf Z \times_{\rm lex} \bf Z$ the totally ordered abelian group that has as domain the cartesian product of $\bf Z \times \bf Z$, the group operations are defined componentwise, and the ordering is the lexicographic ordering. Let $\alg C^{\omega}_{n}$ be the $0$-free reduct of $\Gamma({\bf Z} \times_{\rm lex} {\bf Z}, (n,0))$. Let $\alg C_{\omega}$ be the cancellative hoop defined as the free monoid over one generator $C_{\omega} = \{1 = a^0, a^1, a^2, a^3, \ldots\}$ ordered by $1 = a^0 > a^1 > a^2 > a^3 > \ldots$, and with operations defined as $a^m \cdot a^n = a^{m+n}, a^m \to a^n = a^{\max(n-m,0)}$. Then $\alg C_{\omega}, \alg C^{\omega}_{n}$ are subdirectly irreducible Wajsberg hoops that are not finitely presented, since the former is unbounded and the latter is not simple.

One can obtain other examples of finitely presented Wajsberg hoops using the analogue of \cite[Lemma 3.6]{Mundici11}, which can be proved in a completely analogous way. 
\begin{lemma}\label{lemma:subiso}
Let $f = (f_{1},\ldots, f_{k})$ be a $k$-tuple of Wajsberg functions in $\freeW(n)$, $\scP \sse [0, 1]^{n}$ a pointed rational polyhedron, and $\scQ = f(\scP) \sse[0,1]^k$. Then the subalgebra $\alg W$ of $\W(\scP)$ generated by $f_{1\rest \scP} \ldots f_{k \rest \scP}$ is isomorphic to $\W(\scQ)$.
\end{lemma}
\begin{example}\label{prop:intervals}
 Let $n \geq 1$, thus $\W([\frac{1}{n}, 1]) $, with $[\frac{1}{n}, 1] \sse [0,1]$, is isomorphic to the subalgebra of $\freeW(\{x\})$ generated by $x \to x^{n}$. This is a consequence of Lemma \ref{lemma:subiso}, since the range of $f(x) = x \to x^n$ is exactly $[\frac{1}{n}, 1]$, as depicted in the following Figure \ref{fig:ex}.
\begin{figure}[h]
\begin{center}
\begin{tikzpicture}
\draw [line width = 1pt] (0,0) -- (0,4) -- (4,4) -- (4,0)--(0,0);
\draw [line width = 1pt] (0,4) -- (3,1) --(4,4);
\draw [line width = 0.5pt, dotted] (0,1) -- (3,1);
\draw [line width = 0.5pt, dotted] (3,0) -- (3,1);
\fill (3,0) circle (0.05);
\fill (0,1) circle (0.05);
\node at (3, -0.4) {$\frac{n-1}{n}$};
\node at (-0.3, 1) {$\frac{1}{n}$};
\end{tikzpicture}
\end{center}
\caption{The function $f(x) = x \to x^n$ in $\freeW(1)$.}
\label{fig:ex}
\end{figure}

\end{example}

\section{Projective and exact Wajsberg hoops}
In this section we specialize the duality to two interesting classes of finitely presented Wajsberg hoops, that is, finitely generated projective algebras and exact algebras in $\WH$.

An algebra $\alg P$ is \emph{projective} in a class of algebras $\vv K$ if whenever there are $\alg A, \alg B \in \vv K$, a homomorphism $f: \alg P \to \alg A$ and a surjective homomorphism $g: \alg B \to \alg A$, then there exists $h: \alg P \to \alg B$ such that $f = g \circ h$. This algebraic notion of projectivity corresponds to a categorical notion (where surjective homomorphisms correspond to regular epimorphisms). Equivalently, in a variety of algebras $\vv V$, $\alg P$ is projective if and only if it is a retract of some free algebra $\free_{\vv V}(X)$. We recall that $\alg A$ is a retract of $\alg F$ if there are homomorphisms $i: \alg A \to \alg F$ and $j: \alg F \to \alg A$ such that $j \circ i = id_{\alg A}$. Notice that, necessarily, $i$ is injective and $j$ is surjective. 

An algebra $\alg E$ in a variety $\vv V$ is called \emph{exact} if it is (isomorphic to) a finitely generated subalgebra of a finitely generated free algebra. It follows that finitely generated projective algebras in a variety are a subclass of exact algebras.  

We start by showing that we can indeed use the duality for finitely presented Wajsberg hoops to study exact algebras, and therefore projective algebras, in $\WH$. The following is the analogue of \cite[Corollary 6.6]{Mundici11}, we write the proof here in reference to our previous results for the sake of the reader.
\begin{lemma}\label{lemma:finpressub}
Let $\alg B$ be a finitely generated subalgebra of a finitely presented Wajsberg hoop $\alg A$. Then $\alg B$ is finitely presented.
\end{lemma}
\begin{proof}
Suppose $\alg A \cong \W(\scP)$, for $\scP$ some pointed polyhedron $\scP \sse [0,1]^{n}$. Suppose $\alg B$ is generated by $g_{1}\ldots g_{k} \in \W(\scP)$. Then $g = (g_{1}\ldots g_{k})$ is a pointed $\zz$-map by Lemma \ref{lemma:pzmap}, and its image $g(\scP) \sse [0,1]^{k}$ is a rational polyhedron (\cite[Lemma 3.4]{Mundici11}) and it is pointed because it is the image of a pointed $\zz$-map. By Lemma \ref{lemma:subiso}, $\alg B \cong \W(\scQ)$, thus by Theorem \ref{thm:finpresW} $\alg B$ is finitely presented.
\end{proof}
As an immediate consequence we get the next proposition.
\begin{proposition}\label{prop:projfinpres}
If a $n$-generated Wajsberg hoop is exact (or projective), then it is finitely presented, i.e. isomorphic to $\W(\scP)$ for some pointed rational polyhedron $\scP \sse [0,1]^{n}$.
\end{proposition}
The fact that in a variety being projective is equivalent to being a retract of a free algebra (see \cite{Gh97}), applying the duality of Theorem \ref{thm:duality}, allows us to get the following analogue of \cite[Theorem 1.2]{CM09}. 
Let us call a \emph{pointed $\zz$-retraction} a $\zz$-map $\zeta: [0,1]^{n} \to [0,1]^{n}$ such that  $\zeta^{2} = \zeta$. We call the image of a $\zz$-retraction $\zeta$, $\zeta([0,1]^{n})$, a \emph{pointed $\zz$-retract} of $[0,1]^{n}$. The following is the analogue of \cite[Theorem 1.2]{CM09} and \cite[Corollary 3.6]{MarraSpada13}, and it can be proved in a completely analogous way to the latter as a consequence of the duality in Theorem \ref{thm:duality}.
\begin{theorem}
$\alg A$ is an $n$-generated projective Wajsberg hoop if and only if it is isomorphic to $\W(\scP)$ for a pointed $\zz$-retract $\scP$ of $[0,1]^{n}$.
\end{theorem}
The geometrical study of $\zz$-retracts has been carried out in the papers \cite{CM09,CM12,Ca15}. In particular, necessary conditions for a polyhedron are obtained, that yield a full characterization for the $1$-dimensional case. 

In order to study the analogue of such characterization we need to introduce the notions of contractibility and strong regularity. 
We say that a polyhedron is \emph{contractible} if it is homotopically equivalent to a point. Strong regularity is usually defined as a property of triangulations (see \cite{Mundici11}), however Cabrer shows an intrinsic characterization (\cite[Theorem 4.4]{Ca15}) that we follow here: a polyhedron $\scP$ is \emph{strongly regular} if and only if for each rational vector $v \in \scP$ and each $0 < \delta \in \mathbb{R}$, there exists a rational vector $w \in P$ such that the Euclidean distance ${\rm dist}(v,w) < \delta$, and, calling ${\rm den}(x)$ the least common denominator of the coordinates of a point $x$, ${\rm den}(v)$ and ${\rm den}(w)$ are coprime (in the sense that ${\rm gcd}({\rm den}(v), {\rm den}(w)) = 1$).

Following \cite[Corollary 4.4, Theorem 3.5]{CM12}, if a rational polyhedron $\scP \sse [0,1]^n$ is a $\Z$-retract then it is: contractible, strongly regular, and it contains a point with integer coordinates. In the one-dimensional case the converse also holds.
Since pointed polyhedra always include a vertex of $[0,1]^n$, we get the following.
\begin{theorem}
	If $\scP$ is a pointed $\mathbb{Z}$-retract of $[0, 1]^n$ then
\begin{enumerate}
	\item $\scP$ is contractible,
	\item $\scP$ is strongly regular.
\end{enumerate}
If $\scP$ is one-dimensional, the converse also holds.
\end{theorem}
\begin{proof}
	The proof is a direct consequence of \cite[Theorem 3.5]{CM12}.
\end{proof}
We can then see geometrically that the trivial algebra is projective, which is not surprising since it is the free algebra over the empty set of generators.  Other projective Wajsberg hoops are the algebras in Example \ref{prop:intervals}.

Notice that contractibility implies connectedness, and the polyhedra associated to the algebras $\alg W_{n}$  for $n \geq 2$ are not connected as seen in Example \ref{prop:wn}. 
Since we recall that $\alg C_{\omega}, \alg C^{\omega}_{n}$ are not finitely presented, we obtain the following consequences.
\begin{corollary}
The subdirectly irreducible Wajsberg hoops $\alg C_{\omega}, \alg C^{\omega}_{n}$, $\alg W_{n}$  for $n \geq 2$ are not projective in $\WH$, nor in any variety of residuated lattices that contains $\WH$. 
\end{corollary}
In particular then the $2$-element residuated lattice $\alg W_{2}$ is not projective in $\mathsf{RL}$,
while its bounded version, the $2$-element Boolean algebra, is projective in every subvariety of $\mathsf{FL}$. Indeed, it corresponds to the free algebra on the empty set of generators. 

We can actually obtain stronger negative results. Let us call a Wajsberg hoop $\alg A$ \emph{weakly projective} if whenever it is a homomorphic image of another Wajsberg hoop $\alg B$, it is also isomorphic to one if its subalgebras. 
This notion of weak projectivity is studied in relation to the study of the lattices of subquasivarieties \cite{Gor}.
Clearly, being projective implies being weakly projective, but the opposite is in general not true. Notice that a weakly projective Wajsberg hoop is finitely presented by Lemma \ref{lemma:finpressub}.
\begin{proposition}\label{lemma:weakproj}
Nontrivial finitely generated weakly projective Wajsberg hoops are unbounded.
\end{proposition}
\begin{proof}
Let $\alg A$ be $n$-generated and weakly projective in $\WH$. Then $\alg A$ is a homomorphic image of $\freeW(n)$, and thus isomorphic to a subalgebra of $\freeW(n)$. Since the operations in the free algebra are defined pointwise, and all functions in $\freeW(n)$ are continuous, it follows that the only idempotent element in $\freeW(n)$ is the function constantly equal to $1$. Thus a bounded algebra with more than one element cannot be embedded in $\freeW(n)$, since there is no element in which to map its (necessarily idempotent) lower bound. Thus $\alg A$ is either trivial or unbounded.
\end{proof}
\begin{corollary}\label{cor:unbound}
Nontrivial finitely generated projective Wajsberg hoops are unbounded.
\end{corollary}
Thus no finitely generated positive reduct of an MV-algebra is projective in the variety of Wajsberg hoops. 

We now turn our attention to the larger class of exact algebras, that can be characterized geometrically via the results in \cite{Ca15} and the duality of Theorem \ref{thm:duality}.
\begin{theorem}\label{thm:exact}
	A finitely generated Wajsberg hoop is exact if and only if it is isomorphic to $\W(\scP)$, for a pointed rational polyhedron $\scP \sse [0,1]^n$, for some $n \in \mathbb{N}$, such that:
	\begin{itemize}
		\item $\scP$ is connected;
		\item $\scP$ is strongly regular.
	\end{itemize}
\end{theorem}  
\begin{proof}
Suppose $\alg A$ is a finitely generated exact Wajsberg hoop. Then by definition, $\alg A$ is isomorphic to a finitely generated subalgebra of a (finitely generated) free algebra. By Lemma \ref{lemma:finpressub} $\alg A$ is finitely presented, thus it is isomorphic to $\W(\scP)$ for some pointed rational polyhedron $\scP \sse [0,1]^n$ (Theorem \ref{thm:finpresW}). Via Theorem \ref{thm:duality} and Proposition \ref{prop:dualmorph}, since $\alg A$ embeds into a finitely generated free algebra, there is an onto pointed $\mathbb{Z}$-map from $[0,1]^n$ to $\scP$, for some $n \in \mathbb{N}$. By \cite[Lemma 4.12]{Ca15}, this implies that $\scP$ is connected and strongly regular. 

Vice versa, suppose $\alg A$ is isomorphic to $\W(\scP)$, and $\scP \sse [0,1]^n$ is connected and strongly regular. By \cite[Lemma 4.12]{Ca15}, there is an onto $\mathbb{Z}$-map $\zeta: [0,1]^l \to \scP$ for some $l\in \mathbb{N}$. This  implies that there is also a pointed $\mathbb{Z}$-map from $[0,1]^l$ to $\scP$. Indeed either $\zeta(\one) = \one$, and therefore $\zeta$ is a pointed $\mathbb{Z}$-map, or it suffices to consider $\rho \circ \zeta$, where $\rho$ is the $\mathbb{Z}$-map that rotates the cube and maps $\one$ to the vertex ${\bf v} \in [0,1]^l$ such that $\zeta(\bf v) = \one \in \scP$, thus $\rho \circ \zeta$ is a pointed $\mathbb{Z}$-map. For instance, considering the onto $\Z$-map arising from \cite[Lemma 4.12]{Ca15}, such ${\bf v}$ would be $e_1 = (1, 0, \ldots, 0) \in [0,1]^l$.

Therefore, by the duality in Theorem \ref{thm:duality} and Proposition \ref{prop:dualmorph}, there is an embedding from $\alg A$ to $\freeW(l)$, and thus $\alg A$ is an exact Wajsberg hoop.
\end{proof}

\section{Applications to logic}
We will now use the results we obtained in the previous sections to investigate some interesting analogies and differences between MV-algebras and Wajsberg hoops, or, equivalently, between \luk\ logic and its positive fragment $\lu$. More precisely, we consider the infinite-valued \luk\ logic \textbf{\L} seen as an axiomatic extension of the Full Lambek Calculus with exchange \textbf{FL$_e$}, thus in the language with constants $0, 1$ and binary connectives $\land, \lor, \cdot, \to$ (see \cite{GJKO} for details), and $\lu$ is its $0$-free fragment. If, as usual, we define negation by means of the constant $0$, $\neg x := x \to 0$, notice that $\lu$ is also negation-free.
 
We shall see that, while deducibility in $\WH$ can be reduced to deducibility among positive terms in $\MV$, the same does not hold for admissibility of rules. Moreover, we will prove that while the unification type is the same in $\WH$ and $\MV$, the exact unification type is not. 

We start by showing that quasiequations in Wajsberg hoops can be studied geometrically.
\begin{proposition}\label{prop:geomded}
	Let $t, u$ be terms in the language of Wajsberg hoops over $n$ variables. The following are equivalent:
	\begin{enumerate}
		\item $\WH \models (t \approx 1) \Rightarrow (u \approx 1)$
		\item $O_t \sse O_u$ in $\freeW(n)$
		\item $\MV \models (t \approx 1) \Rightarrow (u \approx 1)$
	\end{enumerate}
\end{proposition}
\begin{proof}
	$(1) \Leftrightarrow (2)$ follows from the fact that the standard Wajsberg hoop on $[0,1]$ generates $\WH$ also as a quasivariety \cite{BlokFerr2000}, which means that a quasiequation $(t \approx 1) \Rightarrow (u \approx 1)$ holds in $\WH$ if and only if it holds in $[0,1]_{\WH}$. Thus, (1) holds if and only if for each valuation 
	$h$ of $\{x_1, \ldots, x_n\}$ to $[0,1]_{\WH}$, $h(t) = 1$ implies $h(u) = 1$, which is equivalent to saying that for all $\x \in [0,1]^n$, $f_t(\x) =1$ implies $f_u(\x) = 1$, where $f_t, f_u \in \freeW(n)$ are the Wajsberg functions associated to $t$ and $u$ respectively. Therefore, (1) and (2) are equivalent.   
	%Since the operations of $\freeW(n)$ are defined componentwise from the ones of the standard algebra, the satisfaction of the quasiequation is equivalent to $(2)$.
	
	Given that $\freeW(n) \sse \freeMV(n)$, $(2) \Leftrightarrow (3)$ is shown in \cite[Corollary 1.10]{Mundici11}.
\end{proof}
As a consequence of algebraizability, we recall that a quasiequation of the kind $(t \approx 1) \Rightarrow (u \approx 1)$ is valid in $\MV$ ($\WH$) if and only if $t$ deduces $u$ in {\bf \L} ({\bf \L}$^+$). Thus, we have also shown that deducibility in $\lu$ corresponds to deducibility of positive terms in \luk\ logic.

In what follows, given a variety $\vv V$ and two sets of identities $\Sigma, \Delta$ in its language, we will write $\Sigma \models_{\vv V} \Delta$ to mean that for each identity $\delta \in \Delta$, the quasiequation $\Sigma \Rightarrow \delta$ is valid in $\vv V$.

 Another direct consequence of the geometrical representation, and another analogy between Wajsberg hoops and MV-algebras, is that the variety $\WH$ is \emph{coherent} in the sense of \cite{Whe}: any finitely generated subalgebra of a finitely presented member is finitely presented (Lemma \ref{lemma:finpressub}). 
 This entails the fact that $\lu$, as \luk\ logic, has \emph{right uniform (Maehara) interpolation} in the sense of \cite{VGMT}:  a variety $\vv V$ admits right uniform Maehara interpolation if for any finite sets of variables $X, Y$ and finite set of equations $\Gamma$ over variables in $X \cup Y$, there exists a finite set of equations $\Pi$ over variables in $Y$ such that 
 for any finite sets of equations $\Sigma$ and $\Delta$ over variables in $Y \cup Z$, $\Gamma,\Sigma \models_{\vv V} \Delta$ if and only if $\Pi, \Sigma \models_{\vv V} \Delta$. 
 \begin{proposition}
 	$\WH$ admits right uniform Maehara interpolation.
 \end{proposition}
\begin{proof}
A variety admits right uniform deductive interpolation if and only if it is coherent and admits deductive interpolation \cite[Proposition 2.4]{KM19}.	 $\WH$ is coherent for Lemma \ref{lemma:finpressub}, while deductive interpolation has been shown in \cite{Mo06}.
Moreover, by \cite[Theorem 3.13]{VGMT}, a variety with right uniform deductive interpolation and the congruence extension property admits Maehara right uniform interpolation, and every commutative variety of residuated lattices has the congruence extension property \cite[Lemma 3.57]{GJKO}.
\end{proof}
As it is clear from the above proof, right uniform Maehara interpolation implies right uniform deductive interpolation (in the sense of \cite[Definition 3.1]{VGMT}). We recall that the Maehara interpolation property for Wajsberg hoops holds for the results in \cite{Mo06, MMT14}. However, as $\ell$-groups, but unlike MV-algebras, $\WH$ does not admit \emph{left uniform deductive interpolation}. A variety $\vv V$ admits left uniform deductive interpolation if for any finite sets $Y, Z$ and finite set of equations $\Delta$ with variables in $Y \cup Z$, there exists a finite set of equations $\Pi$ over variables in $Y$ such that $\Pi \models_{\vv V} \Delta$ and for any set of equations $\Sigma$ over $X \cup Y$, $\Sigma \models_{\vv V} \Delta$ implies $\Sigma \models_{\vv V} \Pi$. This property indeed fails trivially, as in $\ell$-groups (\cite[Remark 4.2]{VGMT}), since there is no set of equations $\Pi(y)$ such that $\Pi \models_{\WH} y \approx z$.

We shall now move to the main applications of the geometrical interpretation of finitely presented Wajsberg hoops, that involve the two related logical problems of unification and admissibility, via the algebraic approaches developed respectively in \cite{Gh97} and \cite{CMe15}.
\subsection{Unification}
Finitely generated projective algebras in a variety that is the equivalent algebraic semantics of a logic are connected to unification problems for their corresponding logics via Ghilardi's algebraic approach to unification \cite{Gh97}. 

The usual notion of unification problem for a logic consists in a finite set of identities, and a solution or \emph{unifier} is a substitution that makes the identities provable in the logic. 

Ghilardi shows that in an algebraizable logic $\mathcal{L}$ with equivalent algebraic semantics $\vv V_{\mathcal{L}}$, one can instead consider the following, equivalent, notion. An algebraic unification problem is a finitely presented algebra $\alg A$ in the variety, and a unifier is a homomorphism $u: \alg A \to \alg P$, where $\alg P$ is a projective algebra in $\vv V_{\mathcal{L}}$. 
The underlying idea is that, given a symbolic unification problem given by the set of identities $\{s_i \approx t_i : i = 1 \ldots n\}$, where $s_i, t_i$ for $i = 1 \ldots n$ are terms in the appropriate language over a set of variables $X$, the corresponding finitely presented algebra is $\free_{\vv V_{\cc L}}(X)/\theta$, where $\theta$ is the congruence (finitely) generated by the pairs $\{(s_i, t_i) : i = 1 \ldots n\}$.
The algebraic unifiers can be ordered by {\em generality} in the following way: given two unifiers for a problem $\alg A$, say $u_{1}: \alg A \to \alg P_{1}$ and $u_{2}: \alg A \to \alg P_{2}$, we say that $u_{1}$ is more general than $u_{2}$ if there exists an homomorphism $h: P_{1} \to P_{2}$ such that $h \circ u_{1} = u_{2}$. This gives a preorder on the set of unifiers, thus one can consider the associated partial order of unifiers that are ``equally general''. 

Then the \emph{unification type} of the problem is said to be: \emph{unitary}, if there is a maximum in the partial order of unifiers; \emph{finitary}, if there are instead finitely many maximal elements; \emph{infinitary}, if instead there are infinitely many maximal elements; \emph{nullary}, otherwise. The unification type of the variety, and via Ghilardi's results of the corresponding logic, is the worst unification type occurring for the unification problems in the variety. 

In \cite{MarraSpada13} the authors show that the unification type of MV-algebras, and hence of \luk\ logic, is nullary. Indeed, using the duality with polyhedra, they show a unification problem with a co-final chain of order-type $\omega$ of unifiers. 

Let us turn our attention to the positive fragment $\lu$, or equivalently, to unification problems in Wajsberg hoops.
First, notice that every unification problem in $\WH$ has a solution, that is, there is always a homomorphism to the projective trivial Wajsberg hoop. Equivalently, the substitution mapping all variables to $1$ unifies every identity. Whereas in \luk\ logic, an identity is unifiable if and only if it is classically satisfiable (\cite[Lemma 3.3]{Je05}), or, from the algebraic point of view, if and only if there is a homomorphism to the $2$-element Boolean algebra, since $\alg 2$ is projective and a retract of every projective MV-algebra \cite{AU21}. 

Moreover, while clearly if an identity is unifiable in $\lu$ it is unifiable in \luk\ as well via the same substitution, unification in $\lu$ does not reduce to unification of positive identities in \luk\ logic, since substitutions in \luk\ logic may also make use of the constant $0$. 
\begin{example}\label{ex:unif}
	Let us define in the language of Wajsberg hoops $x \oplus y : = (x \to (x\cdot y)) \to y$, which in $\MV$ coincides with \luk\ sum, and we write $2x$ for $x \oplus x$. Consider now the identity $2x \to x \approx 1$. Since $\MV \models 2x \to x \approx x \lor \neg x$, the one-set of $2x \to x$ in $[0,1]$ is given by $\{0,1\}$. Thus the identity algebraically corresponds to the finitely presented algebra $\alg W_2$. A unifier for $\alg W_2$ would be a homomorphism to a (finitely generated) projective algebra, that is, a retract of a finitely generated free algebra. Since $\alg W_2$ is bounded, and the only idempotent element of a free finitely generated Wajsberg hoop is the top element, the only possible unifier is the homomorphism to $\alg W_1$.
	
	Equivalently, the identity only has as a solution in $\WH$, (the equivalence class of) the substitution $\sigma(x) = 1$. Whereas in $\MV$, $\tau(x) = 0$ is also a solution, and $\sigma$ and $\tau$ are incomparable in the poset of unifiers.
\end{example}
Nonetheless, we show that we can suitably adapt the pathological example in \cite{MarraSpada13} to our case, and conclude that the unification type of Wajsberg hoops, and thus of the positive fragment of \luk\ logic, is nullary as well.  

The unification problem we consider is equivalent to the one in \cite{MarraSpada13}, with its corresponding polyhedron being the boundary of the unit square in $[0,1]^{2}$, which is a pointed rational polyhedron. In the MV-algebraic case, this corresponds to the identity $x \lor \neg x \lor y \lor \neg y \approx 1$, which is equivalent to the (more cumbersome) identity in the negation-free language of Wajsberg hoops: 
\begin{equation}\label{eq:unifproblem}[((x \to x^{2}) \to x) \to x] \lor [((y \to y^{2}) \to y) \to y] \approx 1
\end{equation}	

Direct computation indeed shows that the one-set of $((x \to x^{2}) \to x) \to x] \lor [((y \to y^{2}) \to y) \to y$ is the border of the unit square, which we will denote by $\cc B$ in analogy to \cite{MarraSpada13}. Notice that this is not a $\zz$-retract of a unit cube, since it is clearly not contractible. The unification problem we are considering is then $\alg A = \freeW(2)/ \cc B$.

We will now find a co-final chain of order-type $\omega$ of unifiers, i.e. projective Wajsberg hoops $\{\alg P_{i}\}_{i \geq 1}$, $i \in \mathbb{N}$, and corresponding homomorphisms $u_{i}: \alg A \to \alg P_{i}$. It will suffice to obtain a slight modification of the key lemma 6.2 in \cite{MarraSpada13}. %Indeed, we can use \cite[Lemma 3.14]{Mundici11}. 

\begin{lemma}\label{lemma:MS}
	For each integer $i\in \mathbb{N}, i \geq 1$, there is an integer $n_i \in \mathbb{N}, n_i \geq 1,$ and a pointed rational polyhedron $\scP_i \sse [0,1]^{n_i}$ that is a pointed $\mathbb{Z}$-retract of $[0, 1]^{n_i}$ such that:
	\begin{enumerate}
		\item\label{lemma:MS1} for each $i\geq 1$ there are onto pointed $\mathbb{Z}$-maps $\mu_i: \scP_i \to \cc{B}$, and injective pointed $\mathbb{Z}$-maps $\nu_i$ such that the following diagram commutes:
		$$
		\xymatrix{
		\cc B & \scP_i\ar[l]_{\mu_i}\ar[d]^{\nu_i}\\
		& \scP_{i+1}\ar[ul]^{\mu_{i+1}}
		}
		$$
		\item\label{lemma:MS2} For any two $i > j > 0$, there is no pointed $\mathbb{Z}$-map $\nu : \scP_i \to \scP_j$ making the following diagram commute.
			$$
		\xymatrix{
		\cc B & \scP_i\ar[l]_{\mu_i}\ar[d]^{\nu}\\
		& \scP_{j}\ar[ul]^{\mu_{j}}
		}
		$$
		\item\label{lemma:MS3} Consider any $n \in \mathbb{N}, n \geq 1$, $\scP \sse [0, 1]^n$ pointed $\mathbb{Z}$-retract of $[0, 1]^n$,  and pointed $\mathbb{Z}$-map $\gamma : \scP \to \cc B$. Then there exist $i_0 \geq 1$ and a $\mathbb{Z}$-map $\gamma' : \scP \to \scP_{i_0}$ such that the following diagram commutes. 
		$$
				\xymatrix{
		\cc B & \scP_{i_0}\ar[l]_{\mu_{i_0}}\\
		& \scP\ar[ul]^{\gamma}\ar[u]_{\gamma'}
		}
		$$
	\end{enumerate}
\end{lemma}
\begin{proof}
	In order to find the pointed polyhedra $\scP_i$, we make use of \cite[Lemma 6.1]{MarraSpada13}. Given each $i \geq 1$, we consider the polyhedron $\mathfrak{t}_i$ as in \cite[Figure 1]{MarraSpada13}, that gives an increasing sequence of squared spirals in $\mathbb{R}^3$ which each project onto the square $\cc B$. Each one of such polyhedra is shown, in \cite[Lemma 6.1]{MarraSpada13}, to be $\mathbb{Z}$-homeomorphic to a rational polyhedron, say $\mathscr{T}_i \sse [0,1]^{n_i}$ where $n_i$ is the number of integer-valued points of $\mathfrak{t}_i$. More precisely, $\mathscr{T}_i$ is the simplicial complex
	given by $$\{\emptyset\} \cup \{e_1, \ldots, e_{n_i}\} \cup \{e_{1,2}, e_{2,3}, \ldots,e_{n_i -1, n_i}\}$$
	where $e_1 \ldots e_{n_i}$ is the standard basis of $\mathbb{R}^{n_i}$ and $e_{i, i+1}$ is the simplex in $\mathbb{R}^{n_i}$ given by the convex hull of $\{e_i, e_{i+1}\}$. 		
	 We can show that each $\mathscr{T}_i$ is $\mathbb{Z}$-homeomorphic to a pointed rational polyhedron, say $\scP_i \sse [0,1]^{n_i}$, where we map each vertex $e_1 \ldots e_{n_i}$ to a different integer-valued point of the same cube, in particular where we map $e_{n_i}$ to $\uno \in [0,1]^{n_i}$. We can indeed consider the $\mathbb{Z}$-homeomorphism $\iota_i$ extending the following assignment on the vertices of $\mathscr{T}_i$:
	 $$
	 \begin{array}{ll}
	 \iota_i(e_1) &= (1, 0, 0, \ldots,0, 0, 0),\\
	 \iota_i(e_2) &= (1, 0, 0, \ldots,0, 0, 1),\\ 
	 \iota_i(e_3) &= (1, 0, 0, \ldots,0, 1, 1),\\ 
		&\ldots\\  
	\iota_i(e_{n_i -2}) &=  (1, 0, 0, 1 \ldots, 1, 1),\\
	\iota_i(e_{n_i -1}) &= (1, 0, 1, \ldots, 1, 1)\\
	 \iota_i(e_{n_i}) &= (1, 1, \ldots, 1, 1).\\ 
	\end{array}
	$$
	 
	 Now, we call $\zeta_i: \mathfrak{t}_i \to \cc B$ the onto projection map,  $f_i$ the $\mathbb{Z}$-homeomorphism from $\mathscr{T}_i$ to $\mathfrak{t}_i$ from \cite[Lemma 6.1]{MarraSpada13}. Direct inspection shows that the composition $\zeta_i \circ f_i \circ \iota_i^{-1}$ is such that $\zeta_i \circ f_i \circ \iota_i^{-1}(\uno) = \zeta_i \circ f_i (e_{n_i}) = \zeta_i (1, 0, n_i) = (1,0) \in [0,1]^2$. We further compose with the $\mathbb{Z}$-homeomorphism $\beta$ that maps $\cc B$ onto itself mapping: $$(1,0) \mapsto (1,1), (1,1) \mapsto (0,1), (0,1) \mapsto (0,0), (0,0) \mapsto (1,0).$$
	 Let us denote with $\eta_i$ the map  from $\mathfrak{t}_i$ to $\mathfrak{t}_{i+1}$ in \cite[Lemma 6.2]{MarraSpada13}.
	 We claim that then the wanted maps are $\mu_i = \beta \circ \zeta_i \circ f_i \circ \iota_i^{-1}$, and $\nu_i = \iota_{i +1} \circ f_{i+1}^{-1} \circ \eta_i \circ f_i \circ \iota_i^{-1}$, as clarified in the following diagram:
	 $$
		\xymatrix{
		\cc B & \cc B\ar[l]_{\beta} & \mathfrak{t}_i \ar[l]_{\zeta_i}\ar[d]^{\eta_i}& \mathscr{T}_i \ar[l]_{f_i} & \scP_i\ar[l]_{\iota_i^{-1}}\\
		& & \mathfrak{t}_{i+1} \ar[ul]^{\zeta_{i+1}} \ar[r]^{f_{i+1}^{-1}}& \mathscr{T}_{i+1} \ar[r]^{\iota_{i+1}} & \scP_{i+1}
		}
		$$
	 Indeed, for (\ref{lemma:MS1}), $\nu_i$ is injective as it is a composition of injective maps, and $\mu_i$ is onto since it is a composition of surjective maps. The diagram commutes since $\zeta_i = \zeta_{i+1} \circ \eta_i$ by \cite[Lemma 6.2 (i)]{MarraSpada13}. For (\ref{lemma:MS2}), if there were such a $\nu$, this would lead to a contradiction of \cite[Lemma 6.2 (ii)]{MarraSpada13}. In the very same way, since a pointed rational polyhedron is in particular a rational polyhedron, the last point follows from \cite[Lemma 6.2 (iii)]{MarraSpada13} and the above $\Z$-homeomorphism between each $\mathscr{T}_i$ and $\scP_i$.
\end{proof}
\begin{theorem}
The unification type of Wajsberg hoops, and therefore of the positive fragment of \luk\ logic, is nullary.
\end{theorem}
\begin{proof}
	The proof follows from Lemma \ref{lemma:MS} and the duality in Theorem \ref{thm:duality}. Indeed, consider the unification problem described by the identity \ref{eq:unifproblem} and represented by the polyhedron $\cc B$. Then, consider each rational polyhedron $\scP_i$, for $i \geq 1$, from Lemma  \ref{lemma:MS}. Each $\scP_i$ is a $\mathbb{Z}$-retract of $[0,1]^{n_i}$, since it is $\mathbb{Z}$-homeomorphic to $\mathscr{T}_i$, and therefore also to $\mathfrak{t}_i$ in \cite[Lemma 6.2]{MarraSpada13}. Since each $\scP_i$ is pointed, the family $\{\scP_i\}_{i \geq 1}$ corresponds to a countable family of projective Wajsberg hoops. Taking the dual of each map $\mu_i$ gives a unifier for the  unification problem represented by $\cc B$. By Lemma \ref{lemma:MS} (\ref{lemma:MS1}) and (\ref{lemma:MS2}), the maps $\mu_i$ for $i \in \mathbb{N}, i \geq 1$ are a strictly increasing chain of order type $\omega$ with respect to the partial order of generality. By Lemma \ref{lemma:MS} (\ref{lemma:MS3}) the chain is co-final. Thus we have a unification problem with nullary type, and therefore the unification type of Wajsberg hoops, and of the positive reduct of \luk\ logic, is nullary.
\end{proof}
\subsection{Admissibility}
Checking the admissibility of rules is a problem closely connected to unification. Given two sets of formulas $\Sigma, \Gamma$ in the language of a variety $\vv V$, we say that $\Sigma \Rightarrow \Gamma$ is $\vv V$-\emph{admissible} if every substitution that unifies all formulas of $\Sigma$ in $\vv V$, it unifies at least one formula of $\Gamma$ in $\vv V$. This is equivalent to the universal formula $\bigwedge\Sigma \Rightarrow \bigvee\Gamma$ being valid in the free algebra over $\omega$ generators in $\vv V$ \cite[Lemma 2.4]{CMe15}. For details on the algebraic approach to admissibility we refer to \cite{CMe15}, 
where the authors show that admissibility is related to what they call \emph{exact unification}. 

More precisely, they define a notion of \emph{exact unification type}, based on a preorder on the set of unifiers for a unification problem that is weaker than the generality preorder. Then the authors prove that, given a rule $\Sigma \Rightarrow \Gamma$, it suffices to check admissibility with respect to substitutions in a set $S$ which includes all the maximal unifiers in the new preorder \cite[Corollary 3.2]{CMe15}. Here we are interested in the algebraic counterpart of this preorder.

Algebraically, studying the exact unification type corresponds to considering the set of \emph{coaxact unifiers}. Given a finitely presented algebra $\alg A \in \vv V$, a coexact unifier is a homomorphism onto an exact algebra $\alg E \in \vv V$. Coexact unifiers are ordered as algebraic unifiers, that is, given two coexact unifiers for a problem $\alg A$, say $u_{1}: \alg A \to \alg E_{1}$ and $u_{2}: \alg A \to \alg E_{2}$, we say that $u_{1}$ is more general than $u_{2}$ if there exists an homomorphism $h: E_{1} \to E_{2}$ such that $h \circ u_{1} = u_{2}$.
 
The study of the exact unification type is relevant since if a variety $\vv V$ has unitary or finitary exact type, and there exists an algorithm for finding a finite complete set of unifiers, then checking admissibility in $\vv V$ is decidable whenever the equational theory of $\vv V$ is decidable. 

While in \cite{CMe15} the authors prove that \luk\ logic has finitary exact type, we will show that Wajsberg hoops, and therefore the positive fragment of \luk\ logic, have instead unitary type. 

The following useful fact is the analogue of \cite[Theorem 4.1, Corollary 4.2]{Je05}, we sketch the proof in our language for the sake of the reader. 
\begin{lemma}\label{lemma:1adm}
All free Wajsberg hoops $\freeW(n)$, for $n \in \mathbb{N}, n \geq 1$, have the same universal theory.
\end{lemma}
\begin{proof}
The statement is equivalent to showing that all free Wajsberg hoops satisfy the same rules $\Sigma \Rightarrow \Delta$. Indeed, every such rule is a universal formula, and vice versa, the validity of a  universal formula is equivalent to the validity of a set of such rules (see the discussion in \cite{Je05}). 

We can show that a rule $\Sigma \Rightarrow \Delta$ is valid in $\freeW(n)$, for any $n \geq 2$, if and only if it is valid in $\freeW(1)$. The left-to-right direction follows from the fact that $\freeW(1)$ is a subalgebra of all $\freeW(n)$, for any $n \geq 2$. 

We now sketch the right-to-left direction. First, one can restrict to considering $\Sigma$ to be one identity $t(x_1, \ldots, x_n) \approx 1$ (\cite[Lemma 3.1]{Ga04}). Let $\Delta = \{u_i(x_1, \ldots, x_n) \approx 1, i = 1\ldots m\}$. Suppose $\Sigma \Rightarrow \Delta$ is not valid in $\freeW(k)$ for some $k \geq 2$. Then there is an assignment $h$ to $\freeW(k)$ such that $h(t) = \top$, where $\top$ is the function constantly equal to $1$, but $h(u_i) < \top$ for all $i= 1 \ldots m$. In particular, there are some rational points ${\bf y_1}, \ldots {\bf y_m} \in [0,1]^k$ where the Wajsberg functions $h(u_i)$ are such that $h(u_i)({\bf y_i})< 1$. 
The idea is to now define a pointed $\mathbb{Z}$-map $f: [0,1] \to [0,1]^k$ such that, if we define a new assignment $k$ of $\{x_1, \ldots, x_n\}$ to $\freeW(1)$ mapping each $x_i$ to $h(x_i) \circ f$, $k$ validates $\Sigma$, but it does not validate any of the identities in $\Delta$.  

%Let us pick $m$ points $a_1, \ldots, a_m$ in $[0,1]$, say $a_i = 1/(i+1)$. 
Let us consider the previous points ${\bf y_1}, \ldots {\bf y_m} \in [0,1]^k$, and let us name their coordinates as follows: ${\bf y_i} = (y_i^1, \ldots, y_i^k)$.
Then, either following a construction similar to the one in the proof \cite[Theorem 4.1]{Je05} or using Schauder's hats (see \cite{CDM}), one can find points $a_1, \ldots, a_m$ in $[0,1]$, and define $k$ McNaughton functions $f_1, \ldots, f_k \in \freeW(1) \sse \freeMV(1)$ such that, for $i = 1, \ldots, m$, $j = 1, \ldots, k$: $$f_j(a_i) = y_i^j,\;\; f_j(1) = 1.$$
Thus, calling $f$ the $\mathbb{Z}$-map that has as components such Wajsberg functions $f = (f_1, \ldots f_k)$, we have that, for $i = 1,\ldots,m$: $$f(a_i) = (y_i^1, \ldots, y_i^k) = {\bf y_i}$$
Let us now consider the assignment $k$ extending $k(x_i) = h(x_i) \circ f$. Then $k$ validates $\Sigma$ (since $h$ does), while it does not validate any of the identities in $\Delta$, since $k(u_i)(a_i) = h(u_i)({\bf y_i})< 1$ by construction, for $i = 1 \ldots m$.

Thus $k$ is an assignment of $\{x_1, \ldots, x_n\}$ in $\freeW(1)$ which testifies the failure of the validity of $\Sigma \Rightarrow \Delta$ in $\freeW(1)$, and the proof is completed.
\end{proof}
For the following lemma and theorem we adapt some ideas in \cite{Je05,Je10,CMe15}. In particular, we use the fact that by \cite[Theorem 4.17]{Ca15}, a rational polyhedron is strongly regular if and only if it is a finite union of anchored sets, where a set $X \sse \mathbb{R}^n$ is called \emph{anchored} if for some $v_1,\ldots,v_m \in \mathbb{R}^n$, $X = \textrm{conv}(v_1,\ldots,v_m)$ and the affine hull of $X$ intersects $\mathbb{Z}^n$. 

Given a set of identities $\Sigma = \{s_i \approx t_i, i = 1 \ldots n\}$, where $s_i, t_i$ are $n$-variable Wajsberg terms, we will write $\free_{\WH}(n)/ \Sigma$ for the quotient of $\free_{\WH}(n)$ with respect to the congruence generated by $\{(s_i, t_i), i = 1 \ldots n\}$.
\begin{lemma}\label{lemma:maxexact}
Let $\Sigma = \{t(x_1, \ldots, x_n) \approx 1\}$ be an identity in the language of Wajsberg hoops. Then there exists a term $u(x_1, \ldots, x_n)$ over the same variables such that, considering $\Delta = \{u(x_1, \ldots, x_n) \approx 1\}$:
\begin{enumerate}
	\item\label{lemma:maxexact1} $\freeW(n)/\Delta$ is an exact Wajsberg hoop. 
	\item\label{lemma:maxexact2} $\WH \models \Delta \Rightarrow \Sigma$;
	\item\label{lemma:maxexact3} $\Sigma \Rightarrow \Delta$ is admissible in $\WH$.
\end{enumerate}
\end{lemma}
\begin{proof}
Let $O_{t} \sse [0,1]^n$ be the one-set of the Wajsberg function corresponding to $t$.

	Then $O_{t}$ can be written as a finite union of polytopes $\scP_i$ for $i =1 \ldots m, m \in \mathbb{N}$:
	$$O_{t} = \bigcup_{i= 1 \ldots m}\scP_i.$$
	Intuitively, we want a connected component of $O_{t}$ that is strongly regular and as large as possible. Thus, we restrict to connected unions of strongly regular polytopes, which means by \cite[Theorem 4.17]{Ca15}, finite unions of anchored polytopes. 
	Let then $$I = \{i \in \{1, \ldots, m\}: \scP_i \mbox{ is anchored}\}, \quad \scP = \bigcup_{i \in I}\scP_i.$$
	Let $\mathscr{C}_1 \ldots \mathscr{C}_k$ be the connected components of $\scP$, and let $\mathscr{C}_{i_0}$ be the connected component that contains $\one$, for some $i_0 \in \{1 \ldots k\}$. 
	Notice that $i_0$ exists, with the limit case being the singleton $\{\one\}$.
	Say that $I_0 \sse I$ is such that  $$\mathscr{C}_{i_0} = \bigcup_{i \in I_0} \scP_i.$$
	Since $\mathscr{C}_{i_0}$ is a pointed rational polyhedron by construction, consider a Wajsberg function $f \in \freeW(n)$ whose one-set is $\mathscr{C}_{i_0}$, and let $u(x_1, \ldots, x_n)$ be a Wajsberg term associated to $f$. We now show (1) -- (3).
	
	(1) By definition, $\mathscr{C}_{i_0}$ is connected and it is strongly regular (since it is a finite union of anchored polytopes, \cite[Theorem 4.17]{Ca15}). Thus $\freeW(n)/\Delta$ is an exact Wajsberg hoop by Theorem \ref{thm:exact}.
	
	(2) By Proposition \ref{prop:geomded}, $\WH \models \Delta \Rightarrow \Sigma$ if and only if $O_{u} \sse O_{t}$, which holds by construction.
	
	(3) By Lemma \ref{lemma:1adm} we can restrict to consider substitutions to terms with $1$ variable, say $y$. Let then $\sigma: \term_{\WH}(x_1, \ldots, x_n) \to \term_{\WH}(y)$ be a $\WH$-unifier of $\Sigma$. Let $\bar \sigma: \freeW(n) \to \freeW(y)$ be the homomorphism induced by $\sigma$, that closes the following diagram, where $\mu, \mu'$ are the natural epimorphisms:
$$
\xymatrix{
\term(x_1, \ldots, x_n) \ar[r]^-\sigma \ar[d]_{\mu} & \term(y) \ar[r]^-{\mu'} & \freeW(y)\\
\freeW(n) \ar[urr]_{\bar \sigma}
}
$$
	Since the image of $\bar\sigma$ is a finitely generated subalgebra of $\freeW(y)$, $\bar\sigma(\freeW(n))$ is an exact Wajsberg hoop, isomorphic to a quotient of the kind $\freeW(y)/\scQ$, where $\scQ \sse [0,1]$ is a pointed polyhedron that is connected and strongly regular, that is, it is a segment of the kind $[q,1]$ where $q \in [0,1]$ is a rational point. From $\bar\sigma$, dually we obtain a strict pointed $\Z$-map $\zeta: [q,1] \to [0,1]^n$ (Proposition \ref{prop:dualmorph}), and $\freeW(n)/\zeta([q,1]) \cong \freeW(y)/[q,1]$. Since $\sigma$ unifies $\Sigma$, $\Sigma \sse ker\bar \sigma$ (\cite[Lemma 2.4]{CMe15}), and thus $\zeta([q,1]) \sse O_t$. Moreover, we can divide $[q,1]$ in subintervals with rational endpoints $\mathscr{S}_1, \ldots \mathscr{S}_k$ such that $\zeta(\mathscr{S}_j) \sse \scP_{i_j}$, for $i_j \in \{1, \ldots m\}$. Indeed, since each $\scP_i \cap \zeta([q,1])$ is a rational polyhedron (\cite[Corollary 2.12]{Mundici11}), each connected component is the image via $\zeta$ of a subinterval with rational endpoints.
	
	Since each $\mathscr{S}_j$ is anchored, so is each $\scP_{i_j}$. Thus, $\zeta[q,1] \sse \scP$, and since it is connected (being the image of a connected interval via a continuous map), and since $\zeta$ is pointed, this implies that $\zeta[q,1] \sse \mathscr{C}_{i_0} = O_{u}$. Since $\zeta$ is the dual of the substitution map acting on the free algebra $\freeW(n)$, this implies that $\Delta \sse ker \bar\sigma$, and thus $\sigma$ also unifies $u$. Therefore, $\Sigma \Rightarrow \Delta$ is admissible in $\WH$.
\end{proof}

\begin{theorem}\label{thm:admissible}
	The exact type of Wajsberg hoops, and of the positive fragment of \luk\ logic, is unitary. 
\end{theorem}
\begin{proof}
	Without loss of generality, let us consider a finitely presented Wajsberg hoop of the kind $\alg A = \freeW(n)/\Sigma$, where $\Sigma = \{t(x_1, \ldots, x_n) \approx 1\}$. Let then $\alg B = \freeW(n)/\Delta$ be the exact Wajsberg hoop given by Lemma \ref{lemma:maxexact}, and let $e: \alg A \to \alg B$ be the dual of the inclusion map from the polyhedron associated to $\alg B$ to the polyhedron associated to $\alg A$ (following the notation of Lemma \ref{lemma:maxexact}, from $\mathscr{C}_{i_0}$ to $O_t$). Then $e$ is a coexact unifier, since it is an homomorphism onto an exact Wajsberg hoop by Lemma \ref{lemma:maxexact} via the duality and Proposition \ref{prop:dualmorph}. We claim that $e$ is maximal in the set of coexact unifiers for $\alg A$. This follows from the admissibility of $\Sigma \Rightarrow \Delta$, Lemma \ref{lemma:maxexact}(\ref{lemma:maxexact3}). Indeed, suppose there is a coexact unifier $u: \alg A \to \alg C$, for some exact Wajsberg hoop $\alg C$ that embeds in $\freeW(\omega)$ via a homomorphism $i$. This, as outlined in the proof of \cite[Theorem 3.6]{CMe15}), gives a unifier for $\Sigma$. Indeed, consider the natural epimorphisms $\mu_1: \freeW(n) \to \alg A = \freeW(n)/\Sigma, \mu_2: \freeW(n) \to \alg B = \freeW(n)/\Delta, \mu: \term_{\WH}(n) \to \freeW(n), \mu_{\omega}: \term_{\WH}(\omega) \to \freeW(\omega)$, and let $\bar \sigma = i \circ u \circ \mu_1$. The substitution is then defined as $\sigma(x_i) = t$ if $\mu_{\omega}(t) = \bar \sigma \circ \mu(x_i)$, thus $\sigma$ unifies $\Sigma$ by construction. Hence, it also unifies $\Delta$, and then $\Delta \sse ker \bar\sigma = ker(u \circ \mu_1)$ since $i$ is injective. Thus there exists an homomorphism $w: \alg B \to \alg C$ that closes the following diagram: 
	$$
\xymatrix{
\freeW(n) \ar[r]^-{\mu_1} \ar[dr]_{\mu_2} \ar@/^2pc/[rrr]^{\bar\sigma}& \alg A \ar[r]^{u} \ar[d]^e& \alg C \ar[r]^-i & \freeW(\omega)\\
&\alg B \ar[ur]_w
}
$$
	That is to say, such that $w \circ \mu_2 = u \circ \mu_1$. Moreover, since $e$ is the dual of the inclusion map, $e \circ \mu_1 = \mu_2$, thus $w \circ e = u$, that is, $u \leq e$ in the preorder of coexact unifiers. Therefore, $e$ is maximal, and since $\alg A$ was an arbitrary finitely presented Wajsberg hoop, the exact type of $\WH$ is unitary.
\end{proof}
Since the proof of Lemma \ref{lemma:maxexact} also shows a procedure to find the maximal coexact unifier of a finitely presented Wajsberg hoop, we get the following consequence. 
\begin{corollary}
	The set of admissible rules is decidable in $\WH$ (and in the positive fragment of \luk\ logic).
\end{corollary}
We stress that, while deducibility of Wajsberg hoops reduces to deducibility in MV-algebras (Proposition \ref{prop:geomded}), this is not true for admissibility. If we consider a 
rule $\Sigma \Rightarrow \Delta$ in the language of the positive reduct that is admissible in $\MV$, this is also admissible in $\WH$, as a consequence of Proposition \ref{prop:geomded}. 
However, the opposite does not hold, as shown by the following example of a rule that is admissible in $\WH$ but not in $\MV$.
\begin{example}
	Consider the quasiequation $(2x \to x \approx 1) \Rightarrow (2x \approx 1)$, where $2x$ is defined as in Example \ref{ex:unif}. Then $\MV \models 2x \to x \approx x \lor \neg x$. The quasiequation is not valid in $\WH$ (nor in $\MV$, equivalently), since $O_{2x \to x} = \{0, 1\} \not\sse O_{2x} = [1/2,1]$ as shown in Figure \ref{fig:exadm}. Moreover, it is not admissible in $\MV$. Indeed, consider the substitution $\sigma(x) = 0$. Clearly $\sigma$ unifies $x \lor \neg x \approx 1$, since $\MV \models 0 \lor \neg 0 \approx 1$, but it does not unify $2x \approx 1$, since $\MV \models 0 \oplus 0 \approx 0$. However, $(2x \to x \approx 1) \Rightarrow (2x \approx 1)$ is admissible in $\WH$. Indeed, as shown by Figure \ref{fig:exadm}, the maximal coexact unifier for $2x \to x \approx 1$ is represented geometrically by the singleton $\{1\}$, which represents the homomorphism onto the trivial algebra. In fact, the only one-variable unifier for $2x \to x \approx 1$ in $\WH$ is (the equivalence class of) the substitution $\sigma(x) = 1$, which unifies any identity written with positive terms.
	\begin{figure}[h]
\begin{center}
\begin{tikzpicture}
\draw [line width = 1pt] (0,0) -- (0,4) -- (4,4) -- (4,0)--(0,0);
\draw [line width = 1.5pt, gray, densely dashed] (0,0) -- (2,3.96) --(4,3.96);
\draw [line width = 1.5pt, black!70, densely dotted] (0, 4) -- (2,2) -- (4,4);
%\draw [line width = 0.5pt, dotted] (2,0) -- (2,4);
%\draw [line width = 0.5pt, dotted] (0,2) -- (2,2);
\node at (0.8, 1) {$g$};
\node at (3, 2.5) {$f$};
%\node at (2, -0.4) {$\frac{1}{2}$};
\end{tikzpicture}
\end{center}
\caption{The functions $f(x) = 2x \to x,\; g(x) = 2x$ in $\freeW(1)$.}
\label{fig:exadm}
\end{figure}

\end{example}

\section{Conclusions and future work}
The duality in terms of rational polyhedra we presented in this manuscript shows to be fruitful to study logical properties of the positive fragment of \luk\ logic, in the same fashion as in the MV-algebraic case. Via this representation, we derived some analogies and differences between \luk\ logic and its positive fragment. 

The duality we presented here is meant to have a clear algebraic interpretation, i.e., the pointed polyhedra are exactly the one-sets of the functions in the free Wajsberg hoop, and finitely presented Wajsberg hoops are the algebras of Wajsberg functions restricted to pointed polyhedra. However, a more general approach could be developed, in the lines of \cite{MarraSpada12,MarraSpada13}. Moreover, the connections with the categorical equivalence in \cite{GT05} in the more general setting shall be investigated.

Given the representation of free BL-algebras (and basic hoops) in \cite{AB10}, the insight given by this work could also be useful to study finitely presented BL-algebras and basic hoops from a geometrical point of view.

\section{Acknowledgements}
The author wishes to thank Paolo Aglian\`o, Tommaso Flaminio, and Vincenzo Marra for their insightful and helpful comments on a previous version of this work. 
	This work has received funding from the European Union's Horizon 2020 research and innovation programme with a Marie Sk\l odowska-Curie grant [890616 to S.U.].
%%%%%%%%%%%%%%%%%%%%%%%%%%%%%%%%%%%%%%%%%%%%
\providecommand{\bysame}{\leavevmode\hbox to3em{\hrulefill}\thinspace}
\providecommand{\MR}{\relax\ifhmode\unskip\space\fi MR }
\providecommand{\MRhref}[2]{%
  \href{http://www.ams.org/mathscinet-getitem?mr=#1}{#2}
}
\providecommand{\href}[2]{#2}

\end{document}